\documentclass[11pt]{articlefederico}

\usepackage{comment}

\usepackage[left=1.5in,top=1.0in,right=1.5in]{geometry}

\usepackage{amsthm}
\usepackage{amsmath}
\usepackage{amssymb}
\usepackage{tikz}
\usepackage{cite}

\usepackage{graphicx}

\newtheorem{theorem}[subsection]{Theorem}
\newtheorem{cor}[subsection]{Corollary}
\newtheorem{lemma}[subsection]{Lemma}
\newtheorem{prop}[subsection]{Proposition}

\newtheorem{quest}[subsection]{Question}

\theoremstyle{remark}
\newtheorem{exa}[subsection]{Example}
\newtheorem{rem}[subsection]{Remark}
\theoremstyle{deff}
\newtheorem{deff}[subsection]{Definition}

\newcommand{\conv}{\mathrm{conv}}
\newcommand{\R}{\mathbb{R}}
\newcommand{\vol}{\mathrm{Vol\,}}

\newcommand{\B}{{\cal B}}

\newcommand{\N}{{\cal N}}
\newcommand{\NN}{{\overline{\N}}}
\newcommand{\Nhat}{{\widehat{N}}}

\begin{document}

\title{\textsf{Lifted generalized permutahedra and \\ composition polynomials.}}
\author{\textsf{Federico Ardila\footnote{\textsf{San Francisco State University, San Francisco, CA, USA, federico@sfsu.edu.}}}
\qquad \textsf{Jeffrey Doker\footnote{\textsf{University of California, Berkeley, Berkeley, CA, USA, jeff.doker@gmail.com.
\newline
This research was partially supported by the National Science Foundation CAREER Award DMS-0956178 (Ardila), the National Science Foundation Grant DMS-0801075 (Ardila), and  the SFSU-Colombia Combinatorics Initiative.}}}}

\date{}

\maketitle

\begin{abstract}
Generalized permutahedra are the polytopes obtained from the permutahedron by changing the edge lengths while preserving the edge directions, possibly identifying vertices along the way. We introduce a ÒliftingÓ construction for these polytopes, which turns an $n$-dimensional generalized permutahedron into an $(n+1)$-dimensional one. We prove that this construction gives rise to Stasheff's multiplihedron from homotopy theory, and to the more general ÒnestomultiplihedraÓ, answering two questions of Devadoss and Forcey.

We construct a subdivision of any lifted generalized permutahedron whose pieces are indexed by compositions. The volume of each piece is given by a polynomial whose combinatorial properties we investigate. We show how this Òcomposition polynomialÓ arises naturally in the polynomial interpolation of an exponential function. We prove that its coefficients are positive integers, and present evidence suggesting that they may also be unimodal.
\end{abstract}

\begin{figure}[h]
\centering
\includegraphics[scale=.32]{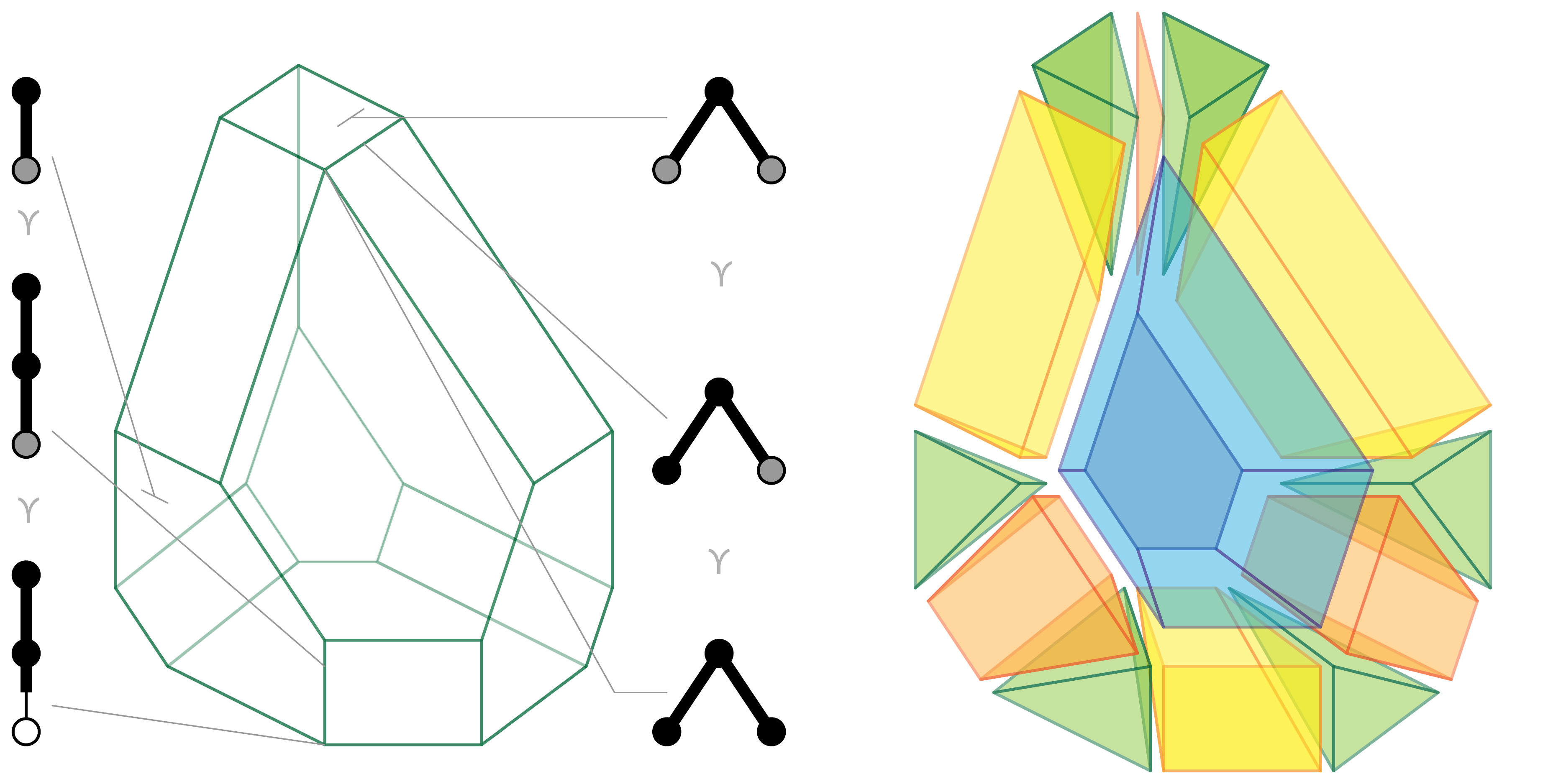} 
\end{figure}


\section{\textsf{Introduction.}}

Generalized permutahedra are the polytopes obtained from the permutahedron by changing the edge lengths while preserving the edge directions, possibly identifying vertices along the way. These polytopes, closely related to polymatroids \cite{Fujishige} and recently re-introduced by Postnikov \cite{Po} have been the subject of great attention due their very rich combinatorial structure. Examples include several remarkable polytopes which naturally appear in homotopy theory, in geometric group theory, and in various moduli spaces: permutahedra, matroid polytopes \cite{BGS}, Pitman-Stanley polytopes \cite{PS}, Stasheff's associahedra \cite{S}, Carr and Devadoss's graph associahedra \cite{CD},  Stasheff's multiplihedra \cite{S}, Devadoss and Forcey's multiplihedra \cite{DF}, and  Feichtner and Sturmfels's and Postnikov's  nestohedra \cite{Po, FS}.

In part 1 of the paper, we introduce a ÒliftingÓ construction which takes a generalized permutahedron $P$ in $\R^n$ into a generalized permutahedron $P(q)$ in $\R^{n+1}$, where $0 \leq q \leq 1$. We show that the lifting construction connects many important generalized permutahedra:

\bigskip

\begin{tabular}{|l|l|}
\hline
 generalized permutahedron $P$ &  lifting $P(q)$  \\
\hline\hline
 permutahedron $P_n$  & permutahedron $P_{n+1}$\\
\hline
 associahedron $\mathcal{K}_n$ & multiplihedron $\mathcal{J}_n$ \\
\hline
 graph associahedron $\mathcal{K}G$&  graph multiplihedron $\mathcal{J}G$ \\
\hline
 nestohedron $\mathcal{K}\B$ & nestomultiplihedron $\mathcal{J}\B$ \\ 
\hline
matroid polytope $P_M$ & independent set polytope $I_M \quad (q=0)$ \\
\hline
\end{tabular}

\bigskip

We provide geometric realizations of these polytopes and concrete descriptions of their face lattices. In particular, we answer two questions of  Devadoss and Forcey: we find the Minkowski decomposition of the graph multiplihedra $\mathcal{J}G$ into simplices, and we construct the nestomultiplihedron $\mathcal{J}\B$.

We also construct a subdivision of any lifted generalized permutahedron $P(q)$ whose pieces are indexed by compositions $c$. The volume of each piece is essentially given by a polynomial in $q$, which we call the \emph{composition polynomial} $g_c(q)$.

\medskip

Part 2 is devoted to the combinatorial properties of the composition polynomial $g_c(q)$ of a composition $c=(c_1, \ldots, c_k)$. We prove that $g_c(q)$ arises naturally in the polynomial interpolation of an exponential function. We also give a combinatorial interpretation of $g_c(q)$ in terms of the enumeration of linear extensions of a poset $P_c$. We prove that $g_c(q) = (1-q)^k f_c(q)$ where $f_c(q)$ is a polynomial with $f_c(1) \neq 0$. We prove that the coefficients of $f_c(q)$ are positive integers. We believe they may be unimodal as well; we have verified this for all 335,922 compositions of at most 7 parts and sizes of parts at most 6.

%
%

\newpage

\noindent 
\begin{Large}
\textsf{PART 1. LIFTED GENERALIZED PERMUTOHEDRA.}
\end{Large}

\bigskip

The first part of the paper is devoted to the lifting construction, which turns an $n$-dimensonal generalized permutahedron $P$ into an $(n+1)$-dimensional one $P(q)$ which depends on a parameter $0 \leq q \leq 1$. 

In Section \ref{section:q-liftings} we introduce the $q$-lifting $P(q)$. We describe its defining inequalities, and its decomposition as a Minkowski sum of simplices. We show that all $q$-liftings $P(q)$ with $0 < q < 1$ are combinatorially isomorphic.

In Section \ref{section:q-faces} we study the face structure of the lifting of $P$. As a warmup, we show that the lifting of the permutahedron $P_n$ is the permutahedron $P_{n+1}$. We then describe the face lattice of $P(q)$ in terms of the face lattice of $P$. 

In Section \ref{section:nestohedra} we begin by recalling Postnikov's and Feichtner-Sturmfels's construction of the nestohedron ${\mathcal K}\B$, and their description of its face lattice in terms of $\B$-forests. We then show that the lifting of ${\mathcal K}\B$ is the nestomultiplihedron ${\mathcal J}\B$, whose face lattice we describe in terms of painted $\B$-forests. As special cases, we see how the multiplihedra ${\mathcal J}_n$ and the graph multiplihedra ${\mathcal J}G$ arise from the lifting construction.

In Section \ref{section: face q-liftings} we give a decomposition of the lifted generalized permutahedron $P(q) \subset \R^n$ whose pieces $P^{\pi}(q)$ are in bijection with the ordered partitions $\pi$ of $[n]$. We show that the volume of  $P^{\pi}(q)$ is essentially given by a polynomial in $q$, which is the subject of study of Part 2 of the paper.

\section{\textsf{Lifting a generalized permutahedron.}}\label{section:q-liftings}

The \emph{permutahedron} $P_n$ is the polytope in $\mathbb{R}^n$ whose $n!$ vertices are the permutations of the vector $(1,2,\dots,n)$. A \emph{generalized permutahedron} is a deformation of the permutahedron, obtained by moving the vertices of $P_n$ in such a way that all edge directions and orientations are preserved, while possibly identifying vertices along the way ~\cite{PRW}.

Postnikov showed \cite{Po} that every generalized permutahedron can be written in the form:
\[
P_n(\{z_I\}) = \left\{(t_1, \ldots, t_n) \in \R^n : \sum_{i=1}^n t_i =
z_{[n]}, \sum_{i \in I} t_i \geq z_I \textrm{ for all } I
\subseteq [n]\large\right\}
\]
where $z_I$ is a real number for each $I \subseteq [n]:=\{1, \ldots, n\}$, and
$z_{\emptyset}=0$. 
The following characterization was announced by Morton et. al. \cite[Theorem 17]{Morton} and Postnikov \cite{Po2}. A complete proof is written down in \cite{AA}; see also \cite[Chapter 44]{Sc}.

\begin{theorem}
\label{thm: Morton}
A set of parameters $\{z_I\}$ defines a generalized permutahedron $P_n(\{z_I\})$ if and only if the $z_I$ satisfy the supermodular inequalities for all $I, J \subseteq [n]$:
\[
z_I + z_J \leq z_{I\cup J} + z_{I\cap J}.
\]
\end{theorem}

\begin{rem}\label{rem:shift}
By performing a parallel shift, we will assume that all our generalized permutahedra are in the positive orthant. In particular, this implies that $z_I \geq 0$ for all $I \subseteq [n]$, and that $z_I \leq z_J$ for $I \subseteq J \subseteq [n]$.
\end{rem}

We now introduce \emph{lifting}, a procedure which converts a generalized permutahedron in $\mathbb{R}^n$ into a \emph{lifted generalized permutahedron} in $\mathbb{R}^{n+1}$.

\begin{deff}
\label{prop: z_I}
Given a generalized permutahedron $P=P_n(\{z_I\})$ in $\mathbb{R}^n$ and  a number $0 \leq q \leq 1$, let the \emph{$q$-lifting} of $P$ be the polytope $P(q)$ given by the inequalities
\[
\sum_{i=1}^{n+1} t_i = z_{[n]},  \qquad
 \sum_{i \in I} t_i \geq qz_I \textrm{ for } I \subseteq [n], \qquad
\sum_{i \in I \cup \{n+1\}} t_i \geq z_I \textrm{ for  } I \subseteq [n]. 
\]
In other words, $P(q):=P_{n+1}(\{z'_I\})$ where $z'_J = qz_J$ and $z'_{J\cup \{n+1\}} = z_J$ for $J\subseteq [n]$. The polytope $P(q)$ is called a \emph{lifted generalized permutahedron}.
%

We will let the \emph{lifting} of $P$ refer to any $q$-lifting with $0<q<1$. We will see in Corollary \ref{cor:lifting} that all such $q$-liftings are combinatorially isomorphic.
\end{deff}

\begin{prop}
If $P$ is a generalized permutahedron, then its $q$-lifting $P(q)$ is a generalized permutahedron. 
\end{prop}

\begin{proof}
Keeping Remark \ref{rem:shift} in mind, one easily checks that the hyperplane parameters $\{z'_I\}_{I\subseteq[n+1]}$ are supermodular.
\end{proof}

Notice that the $1$-lifting $P(1)$ is the natural embedding of $P$ in the hyperplane $x_{n+1}=0$ of $\mathbb{R}^{n+1}$. The $0$-lifting $P(0)
=P_{n+1}(\{z'_I\})$ is the generalized permutahedron in $\R^{n+1}$ defined by $z'_J = 0$ and $z'_{J\cup\{n+1\}} = z_J$ for all $J\subseteq[n]$. 

Recall that the \emph{Minkowski sum} of two polytopes $P$ and $Q$ in $\R^n$ is defined to be $P+Q:=\{p+q \, : \, p \in P, q \in Q\}$. The hyperplane parameters $\{z_I\}$ of generalized permutahedra are additive with respect to Minkowski sums \cite{ABD, Po}, so we have:

\begin{prop} \label{prop:q-lift}
For $0 \leq q \leq 1$, the $q$-lifting of any generalized permutahedron $P$ satisfies $P(q) = qP(1) + (1-q)P(0)$.
\end{prop}

\begin{figure}[h]
\centering
\includegraphics[scale=.25]{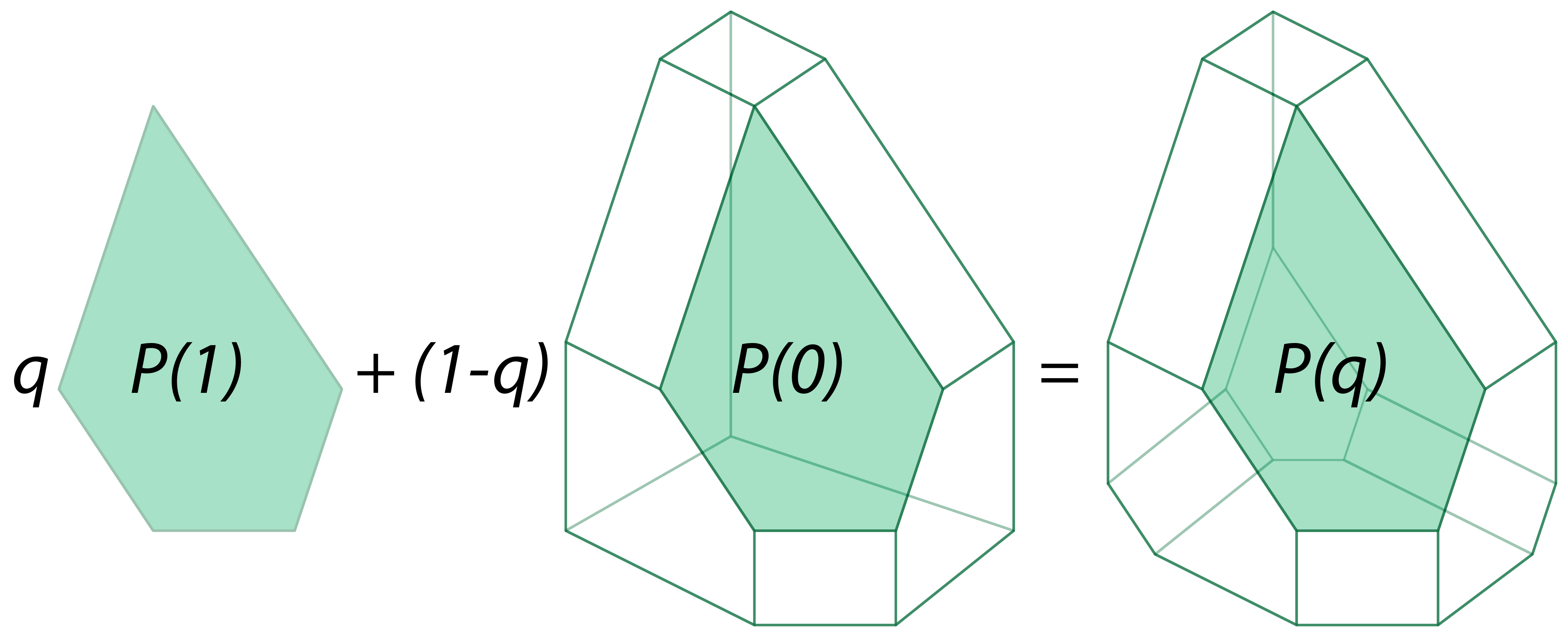} 
\caption{The $q$-lifting of a generalized permutahedron $P_n(\{y_I\})$, shown projected onto the $3$-dimensional hyperplane $x_4=0$. \label{fig:qlift}}
\end{figure}

\begin{cor}\label{cor:lifting}
All $q$-liftings of $P$ with $0 < q < 1$ are combinatorially isomorphic.
\end{cor}

\begin{proof}
By Proposition \ref{prop:q-lift}, the normal fan of $P(q)$ is the common refinement of the normal fans of $P(0)$ and $P(1)$.
\end{proof}

For each $I \subseteq [n]$, consider the simplex $\Delta_I = \conv\{e_i \, : \, i \in I\}$. Any generalized permutahedron  $P=P_n(\{z_I\})$ can be written uniquely as a signed Minkowski sum of simplices in the form $P=P_n(\{y_I\}):=\sum y_I \Delta_I$ for $y_I \in \mathbb{R}$.\footnote{An equation like $P-Q=R$ should be interpreted as $P=Q+R$.} \cite{ABD, Po} The $z$-parameters and the $y$-parameters of $P$ are linearly related by the equations
\[
z_I = \sum_{J \subseteq I} y_J, \qquad \textrm{ for all } I \subseteq [n].
\]

\begin{prop} \label{prop:P(q)linear}
The $q$-lifting of the generalized permutahedron $P = \sum_I y_I \Delta_I$ is
\[
P(q) = q \sum_I y_I \Delta_I + (1-q) \sum_I y_I \Delta_{I \cup \{n+1\}}.
\]
%
\end{prop}
\begin{proof}
This follows directly from the linear relation between the $z_I$ and the $y_I$. 
\end{proof}
%
%

From these observations it follows that the face of $P(q)$ maximized in the  direction $(1,\dots,1,0)$ is a copy of $P$, while the face maximized in the opposite direction is a copy of $P$ scaled by $q$. The vertices of $P(q)$ will come from vertices of $P$, with a factor of $q$ applied to certain specific coordinates. We describe them in Section \ref{section: face q-liftings}.

%
%

\section{\textsf{Faces of lifted generalized permutahedra.}}\label{section:q-faces}

We now look into the face structure of lifted generalized permutahedra. An important initial observation is that their face lattices are always coarsenings of the face lattice of the permutahedron $P_n$ \cite{Morton, Po, PRW}. 

\begin{deff}
Consider the linear functional $f(x_1,\dots,x_n) = a_1x_1+\cdots+a_nx_n$. We partition $[n]$ into blocks $\pi_1,\dots,\pi_k$ such that $a_i=a_j$ if and only if $i$ and $j$ both belong to the same block $\pi_s$, and $a_i<a_j$ if and only if $i\in \pi_s$ and $j\in \pi_t$ for some $s<t$. If we let $\pi = \pi_1|\cdots|\pi_k$ 
then we say that the functional $f$ is of \emph{type} $\pi$. Slightly abusing notation, we write $f(x)$ as $f_\pi(x)$. For a generalized permutahedron $P$ in $\R^n$, the face of $P$ maximizing $f$ only depends on $\pi$, and we call it 
$P_\pi$.
\end{deff}

The following properties of the maximal face $P_\pi$ are known  \cite{Morton, Po, PRW}  and will be very important to us throughout the paper:

\begin{prop}\label{prop:props}
Let $\pi=\pi_1|\cdots|\pi_k$ be an ordered partition of $[n]$.
\begin{enumerate}
\item For a subset $I$ of $[n]$, the $\pi$-maximal face of the simplex $\Delta_I$ is $(\Delta_I)_\pi = \Delta_{I \cap \pi_{j(I)}}$, where $j(I) = \max\{j \, : \, I \cap \pi_j \neq \emptyset\}$. 
\item
The $\pi$-maximal face of the generalized permutahedron $P_n(\{y_I\}) = \sum_{I \subseteq [n]} y_I\Delta_I$ is
$ \left(\sum_{I \subseteq [n]} y_I\Delta_I\right)_\pi = 
\sum_{I \subseteq [n]} y_I\Delta_{I \cap \pi_{j(I)}}$.
\item
The $\pi$-maximal face of the generalized permutahedron $P_n(\{z_I\})$
is $(P_n(\{z_I\}))_\pi = P_1 \times \cdots \times P_k$, where $P_1 \in \R^{\pi_1}, \ldots, P_k \in \R^{\pi_k}$ are the generalized permutahedra $P_j = P(\{z_I^j\}_{I \subseteq \pi_j})$ given by $z_I^j = z_{\pi_1\cup \cdots \cup  \pi_{j-1} \cup I} - z_{\pi_1\cup \cdots \cup \pi_{j-1}}$ for $I \subseteq \pi_j$.
\end{enumerate}
\end{prop}

\begin{proof}
The first statement is clear, and the second one is implied by the fact that $(P+Q)_\pi = P_\pi + Q_\pi$ for any polytopes $P$ and $Q$. The third statement follows since $(P_n(\{z_I\}))_\pi$ consists of the points $x \in P_n(\{z_I\})$ such that $\sum_{i \in \pi_j}x_i = z_{\pi_1\cup \cdots \cup \pi_j} - z_{\pi_1\cup \cdots \cup \pi_{j-1}}$ for all $1 \leq j \leq k$.
\end{proof}

Recall that the face lattice $\mathcal{L}(P_n)$ of the permutahedron $P_n$ is isomorphic to the poset $(\mathcal{P}^n,\prec)$, where $\mathcal{P}^n$ is the set of all ordered partitions of the set $[n]$, and $\pi \prec\pi'$ if and only if $\pi'$ coarsens $\pi$ \cite{Zi}. First we show that the $q$-lifted permutahedron $P_n(q)$ is combinatorially equivalent to $P_{n+1}$. 

\begin{prop}
The lifting of the permutahedron $P_n$ is combinatorially equivalent to the permutahedron $P_{n+1}$.
\end{prop}
\begin{proof}
By definition $P_n(q)$ is a generalized permutahedron in $\mathbb{R}^{n+1}$, and hence its face lattice is a coarsening of the poset of ordered partitions on a set of size $n+1$. We will show that this coarsening is trivial; \emph{i.e.}, that every strict containment of faces in $P_{n+1}$ corresponds to a strict containment of faces in $P_n(q)$. 

The permutahedron $P_n$ is a zonotope, and it can be represented as the Minkowski sum of all coordinate 1-simplices $\Delta_{ij}$ for $1\leq i<j\leq n$. Using our established notation, we write $P_n = P_n(\{y_I\})$ where $y_I=1$ if $I$ has size $2$, and $y_I=0$ otherwise. Let $\pi = \pi_1|\cdots|\pi_k$ be an ordered partition of $[n+1]$, and let $P_n(q)_\pi$ be the corresponding maximal face of $P_n(q)$.
It suffices to show that any minimal coarsening $\sigma$ of $\pi$, obtained by joining blocks $\pi_i$ and $\pi_{i+1}$, leads to a different maximal face $P_n(q)_\sigma$.

For every pair $b_1, b_2\in [n+1]$ the Minkowski decomposition of $P_n(q)$ contains a simplex with $\Delta_{b_1 b_2}$ as a face. Take $b_1\in \pi_i$ and $b_2\in \pi_{i+1}$. Then the Minkowski decomposition of the face $P_n(q)_\sigma$ includes a one-dimensional contribution from $\Delta_{b_1 b_2}$, whereas the decomposition of $P_n(q)_\pi$ does not. Thus $P_n(q)_\pi$ is properly contained in $P_n(q)_\sigma$, as we wished to show. 
\end{proof}

Now we extend our focus to face lattices of general $q$-liftings. In the remainder of this section, we assume for simplicity that the generalized permutahedra $P$ we are analyzing are contained in the positive orthant $\mathbb{R}_{>0}^n$. 

\begin{deff}
Let $P$ be a generalized permutahedron in $\mathbb{R}^n$, and let $\pi$ and $\mu$ be ordered partitions of $[n]$. Then we say that $\pi \sim \mu$ if $P_{\pi} = P_{\mu}$. 
We can write the face lattice of $P$ as
\[
\mathcal{L}(P) \cong \left({\mathcal{P}^n},\prec\right)/\sim.
\]
The order $\prec$ 
on equivalence classes is as follows: the equivalence class $[\mu]$ covers $[\pi]$ if and only if 
there exist $\pi'\in[\pi]$ and $\mu'\in[\mu]$ such that 
$\mu'$ coarsens $\pi'$.
\end{deff}

We now describe the equivalence relation $\sim'$ on $\mathcal{P}^{n+1}$ induced by $P(q)$ in terms of the equivalence relation $\sim$ on $\mathcal{P}^{n}$ induced by $P$.

\begin{deff} \label{def: face lattice}
Let $\pi'$ and $\mu'$ be ordered partitions of $[n+1]$, and let $\pi=\pi_1|\cdots|\pi_k$ and $\mu=\mu_1|\cdots|\mu_l$ be the partitions of $[n]$ obtained by deleting $n+1$ from $\pi'$ and $\mu'$ respectively (and deleting the resulting empty block if $n+1$ was alone in its block). Let $a$ and $b$ be the indices of the blocks of $\pi$ and $\mu$ containing $n+1$, respectively. Then we say that $\pi' \sim' \mu'$ if $\pi \sim \mu$, $\pi_a = \mu_b$, and $\bigcup_{i>a}\pi_i = \bigcup_{i>b}\mu_i$.
\end{deff}

%
%

\begin{prop}\label{prop: face lattice}
Let $P$ be a generalized permutahedron in the positive orthant $\mathbb{R}_{>0}^n$. Using the notation established above, the face lattice of $P(q)$ is given by
\[
\mathcal{L}(P(q)) \cong \left({\mathcal{P}^{n+1}},\prec\right)/\sim'.
\]
\end{prop}
\begin{proof}
Write $P = P_n(\{y_I\})$. 
Assume that $P(q)_{\pi'} = P(q)_{\mu'}$ for two ordered partitions $\pi'$ and $\mu'$ of $[n+1]$. As bejore, let $j(I)$ (\emph{resp.} $k(I)$) be the largest $j$ (\emph{resp.} $k$) such that $I$ intersects $\pi'_j$ (\emph{resp.} $\mu'_k$).
 By Proposition \ref{prop:P(q)linear},
\begin{eqnarray*}
P(q)_{\pi'} &=& q\sum_{I\subseteq [n]}y_I({\Delta_I})_{\pi'} + (1-q)\sum_{I\subseteq [n]}y_I(\Delta_{I\cup\{n+1\}})_{\pi'} \\
&=& q\sum_{I\subseteq [n]}y_I({\Delta_I})_{\pi} +(1-q)\sum_{I :  j(I) > a }y_I({\Delta_I})_{\pi}+ (1-q)\sum_{I  :  j(I) \leq a } y_I \Delta_{(I \cap \pi_a) \cup \{n+1\}} \\
&=& 
\sum_{I  :  j(I) > a }y_I({\Delta_I})_{\pi}
+ q\sum_{I  :  j(I) \leq a } y_I ({\Delta_I})_{\pi}
+ q\sum_{I  :  j(I) \leq a } y_I \Delta_{(I \cap \pi_a) \cup \{n+1\}},
  \end{eqnarray*}
and similarly for $P(q)_{\mu'}$. If we have $P(q)_{\pi'} = P(q)_{\mu'}$ for \textbf{one} choice of $q$ with $0<q<1$, then $\pi'$ and $\mu'$ are in the same cone of the normal fan of $P(q)$, which does not depend on $q$ as argued in Corollary \ref{cor:lifting}.
It follows that $P(q)_{\pi'} = P(q)_{\mu'}$ for \textbf{any} $q$ with $0<q<1$. Since the first summand does not involve $q$ and only the last summand involves the $(n+1)$-st coordinate, we must have
\begin{eqnarray*}
\sum_{I  :  j(I) > a }y_I({\Delta_I})_{\pi} &=& \sum_{I  :  k(I) > b }y_I({\Delta_I})_{\mu},\\
\sum_{I  :  j(I) \leq a }y_I({\Delta_I})_{\pi} &=& \sum_{I  :  k(I) \leq b }y_I({\Delta_I})_{\mu},\\
\sum_{I  :  j(I) \leq a }y_I\Delta_{I \cap \pi_a} &=& \sum_{I  :  k(I) \leq b }y_I\Delta_{I \cap \mu_b}. 
\end{eqnarray*}
Adding the first two equations gives $P_\pi = P_\mu$, so $\pi \sim \mu$. Since $P(q)_{\pi'} = P(q)_{\mu'}$ has full support, the polytope described by the first equation has support $\bigcup_{i>a}\pi_i  = \bigcup_{i>b}\mu_i$, while the one described by the third equation has support $\pi_a = \mu_b$. It follows that $\pi' \sim' \mu'$. The converse follows similarly.
%
\end{proof}

\section{\textsf{Nestohedra and nestomultiplihedra.}}\label{section:nestohedra}

In his work on homotopy associativity for $A_\infty$ spaces, Stasheff \cite{S} defined the \emph{multiplihedron} ${\mathcal J}_n$, a cell complex which has since been realized in different geometric contexts by Fukaya, Oh, Ohta, and Ono \cite{FOOO}, by Mau and Woodward \cite{MW}, and others. It was first realized as a polytope by Forcey \cite{F}.

More generally, Devadoss and Forcey \cite{DF} defined, for each graph $G$, the \emph{graph multiplihedron} ${\mathcal J}G$. This is a polytope related to the graph associahedron ${\mathcal K}G$. \cite{ARW, CD} When $G$ has no edges, they gave a description of ${\mathcal J}G$ as a Minkowski sum. They asked for a Minkowski sum description of ${\mathcal J}G$ for arbitrary $G$.

In a different direction, Postnikov \cite{Po} defined the \emph{nestohedron} ${\mathcal K}\B$, an extension of  graph associahedra to the more general context of building sets $\B$. Devadoss and Forcey asked whether there is a notion of \emph{nestomultiplihedron} ${\mathcal J}{\B}$, which extends the graph multiplihedra to this context.

In this section we answer these questions affirmatively in a unified way, by showing that the $q$-lifting of the graph associahedron ${\mathcal K}G$ is the graph multiplihedron ${\mathcal J}G$ and, more generally, the $q$-lifting of the nestohedron ${\mathcal K}\B$ is the desired nestomultiplihedron ${\mathcal J}{\B}$.

\subsection{\textsf{Nestohedra and $\B$-forests.}}

\begin{deff} \cite{FS, Po}
A \emph{building set} $\B$ on a ground set $[n]$ is a collection of subsets of $[n]$ such that: \\
(B1) If $I, J \in \B$ and $I \cap J \neq \emptyset$ then $I \cup J \in \B$.\\
(B2) For every $e \in [n]$, $\{e\} \in \B$.
\end{deff}

An important example is the following: given a graph $G$ on a vertex set $[n]$, the associated building set $\B(G)$ consists of the subsets $I \subseteq [n]$ for which the induced subgraph $G|_I$ is connected. Such subsets are sometimes called the \emph{tubes} of $G$.

If $\B$ is a building set on $[n]$ and $A \subseteq [n]$, define the \emph{induced building set} of $\B$ on $A$ to be $\B|_A:= \{I \in \B \, : \, I \subseteq A\}$. Also let $\B_{\max}$ be the set of containment-maximal elements of $\B$.

\begin{deff} \cite{FS, Po}
A \emph{nested set} $\N$ for a building set $\B$ is a subset $\N \subseteq \B$ such that: \\
(N1) If $I, J \in \N$ then $I \subseteq J$ or $J \subseteq I$ or $I \cap J = \emptyset$. \\
(N2) If $J_1, \ldots, J_k \in \N$ are pairwise incomparable and $k \geq 2$ then $J_1 \cup \cdots \cup J_k \notin \B$. \\
(N3) $\B_{\max} \subseteq  \N$. \\
The \emph{nested set complex} $\N(\B)$ of $\B$ is the simplicial complex on $\B$ whose faces are the nested sets of $\B$.
\end{deff}

When $\B(G)$ is the building set of tubes of a graph, the nested sets are called the \emph{tubings} of $G$. If $G$ is the graph shown in Figure \ref{fig:graph}(a), an  example of a nested set or tubing is $\N = \{3,4,6,7,379,48,135679, 123456789\}$, shown in Figure \ref{fig:graph}(b).\footnote{We omit the brackets from the sets in $\N$ for clarity.}

\begin{figure}[h]
\centering
\includegraphics[scale=.2]{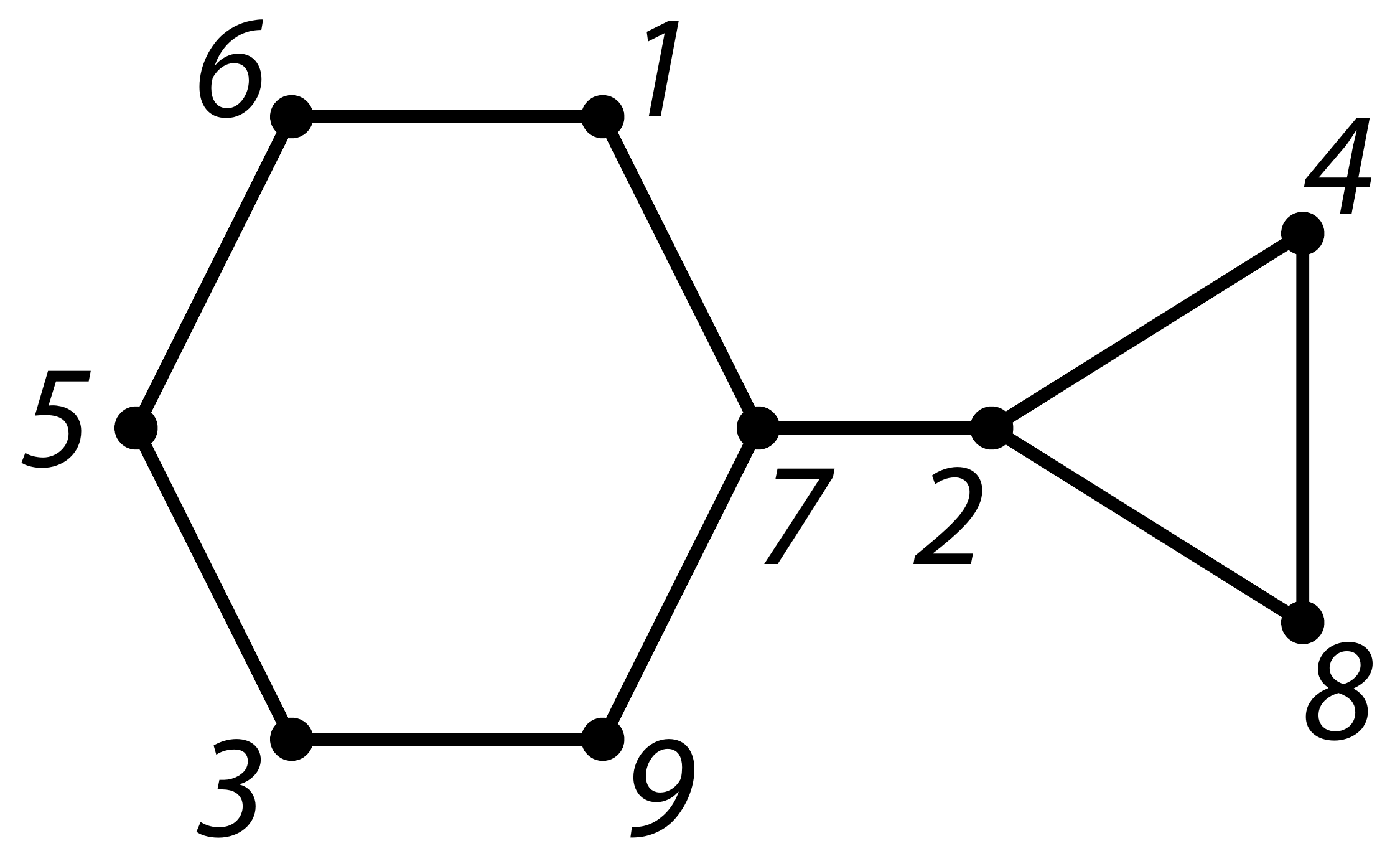}\qquad \qquad \includegraphics[scale=.2]{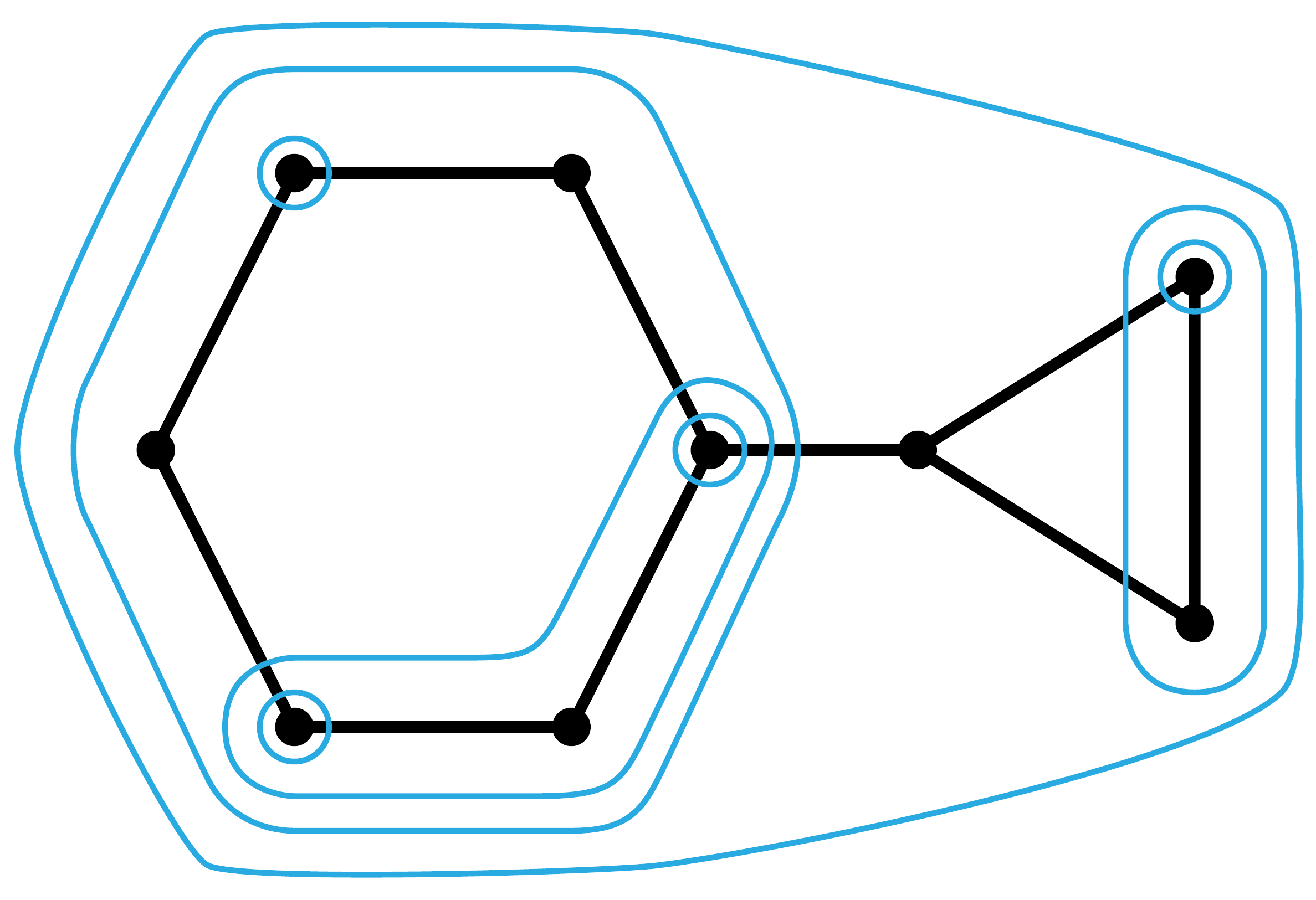} 
\caption{(a) A graph $G$. (b) A nested set or tubing of $G$.\label{fig:graph}}
\end{figure}

The sets in a nested set $\N$ form a poset by containment. This poset is a forest rooted at $\B_{\max}$ by (N1). Relabelling each node $N$ with the set $\Nhat := N \backslash \bigcup_{M \in \N \, : \, M < N} M$, we obtain a $\B$-forest:

\begin{figure}[h]
\centering
\includegraphics[scale=.3]{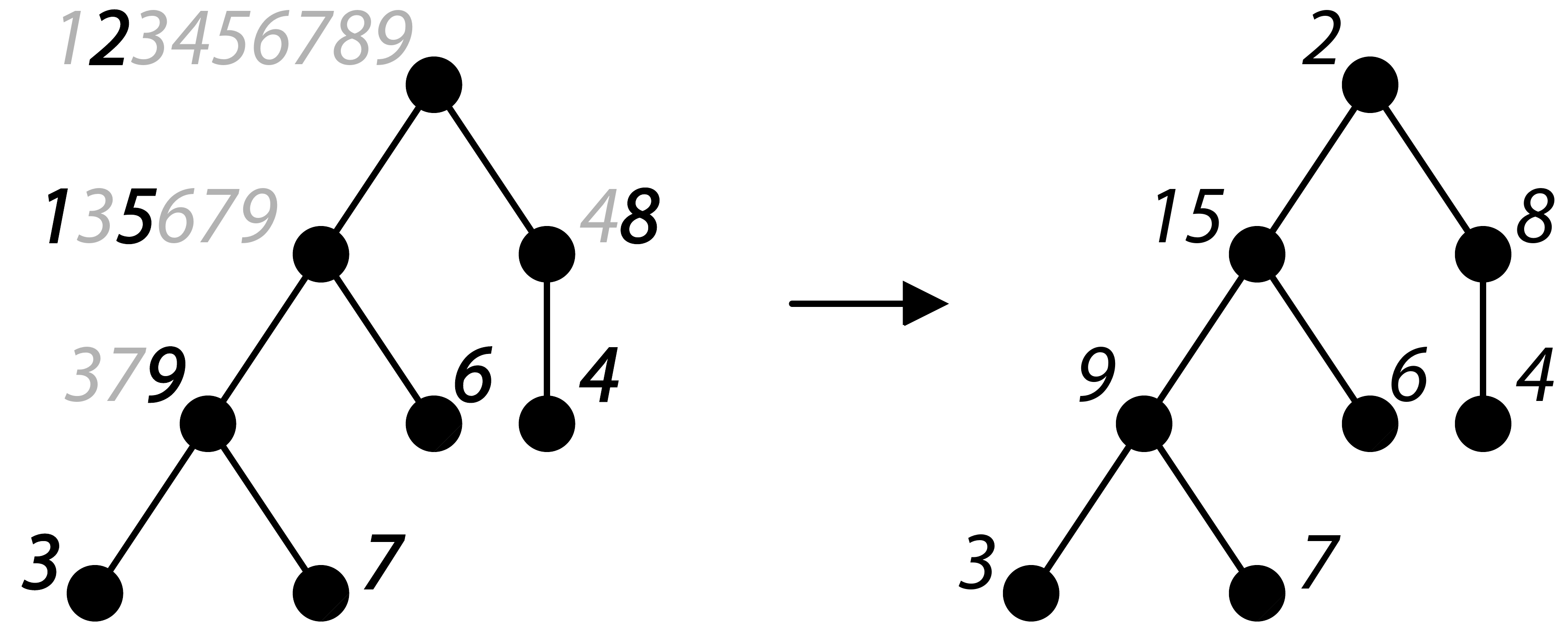}
\caption{The poset and the $\B$-forest for the nested set $\N = \{3,4,6,7,379,48,135679, 123456789\}$ of Figure \ref{fig:graph}(b).
\label{fig:B-forest}}
\end{figure}

\begin{deff} \cite{FS, Po}
Given a building set $\B$ on $[n]$, a \emph{$\B$-forest} $\N$ is a rooted forest whose vertices are labeled with non-empty sets 
partitioning $[n]$ such that: \\
(F1) For any node $S$,  $\N_{\leq S} \in \B$. \\
(F2) If $S_1, \ldots, S_k$ are incomparable and $k \geq 2$, $\bigcup_{i=1}^k \N_{\leq S_i} \notin \B$. \\
(F3) If $R_1, \ldots, R_r$ are the roots of $F$, then the sets $\N_{\leq R_1}, \ldots, \N_{\leq R_r}$ are precisely the maximal elements of $\B$. 
\end{deff}

Here $\N_{\leq S} := \bigcup_{T \leq S} T$. It is clear from the definitions that nested sets for $\B$ are in bijection with $\B$-forests. As the notation suggests, we will make no distinction between a nested set and its corresponding $\B$-forest.

Given a $\B$-forest $\N$, the \emph{contraction} of an edge $ST$ (where $T$ is directly above $S$ in the forest) is obtained by removing the edge $ST$, and relabeling the resulting merged vertex with the set $S \cup T$. Containment of nested sets corresponds to successive contraction of $\B$-forests. Say $\N \geq \N'$ if the nested set $\N'$ is contained in the nested set $\N$ or, equivalently, if the $\B$-forest $\N'$ is obtained from the $\B$-forest $\N$ by a series of successive contractions. 
Then we have:

\begin{theorem}\label{th:nest}\cite{FS,Po}
The face poset of the \emph{nestohedron} 
\[
{\mathcal K}\B:= \sum_{B \in \B} \Delta_B
\]
is isomorphic to the opposite of the poset of $\B$-forests.
\end{theorem}

The nested set complex $\N(\B)$ is a cone over $\B_{max}$, and Theorem \ref{th:nest} says that the link of $\B_{max}$, called the \emph{reduced nested set complex}, is combinatorially dual to the nestohedron. 

In Theorem \ref{th:nesto} we will prove a ``painted" version of this result, following a similar proof strategy. In fact one can deduce Theorem \ref{th:nest}  directly from Theorem \ref{th:nesto}, as we will explain in Remark \ref{rem:nest}.

It is worth remarking that the \emph{graph associahedron} ${\mathcal K}G$ is the nestohedron for the building set $\B(G)$ of the graph $G$. For instance, if $G=P_n$ is the path with $n$ vertices, then $\B(P_n) = \{[i,j] \, : \, 1 \leq i \leq j \leq n\}$ is the nested set of intervals, and the resulting nestohedron is the \emph{associahedron} ${\mathcal K}_n$. Figure \ref{fig:assoc} illustrates this in the case $n=3$. There is a simple bijection between $\B(P_3)$-forests and planar trees on $n+1$ unlabeled leaves. A ``painted" version of this bijection is illustrated in Figure \ref{fig:treeflip}. 

\begin{figure}[h]
\centering
\includegraphics[scale=.4]{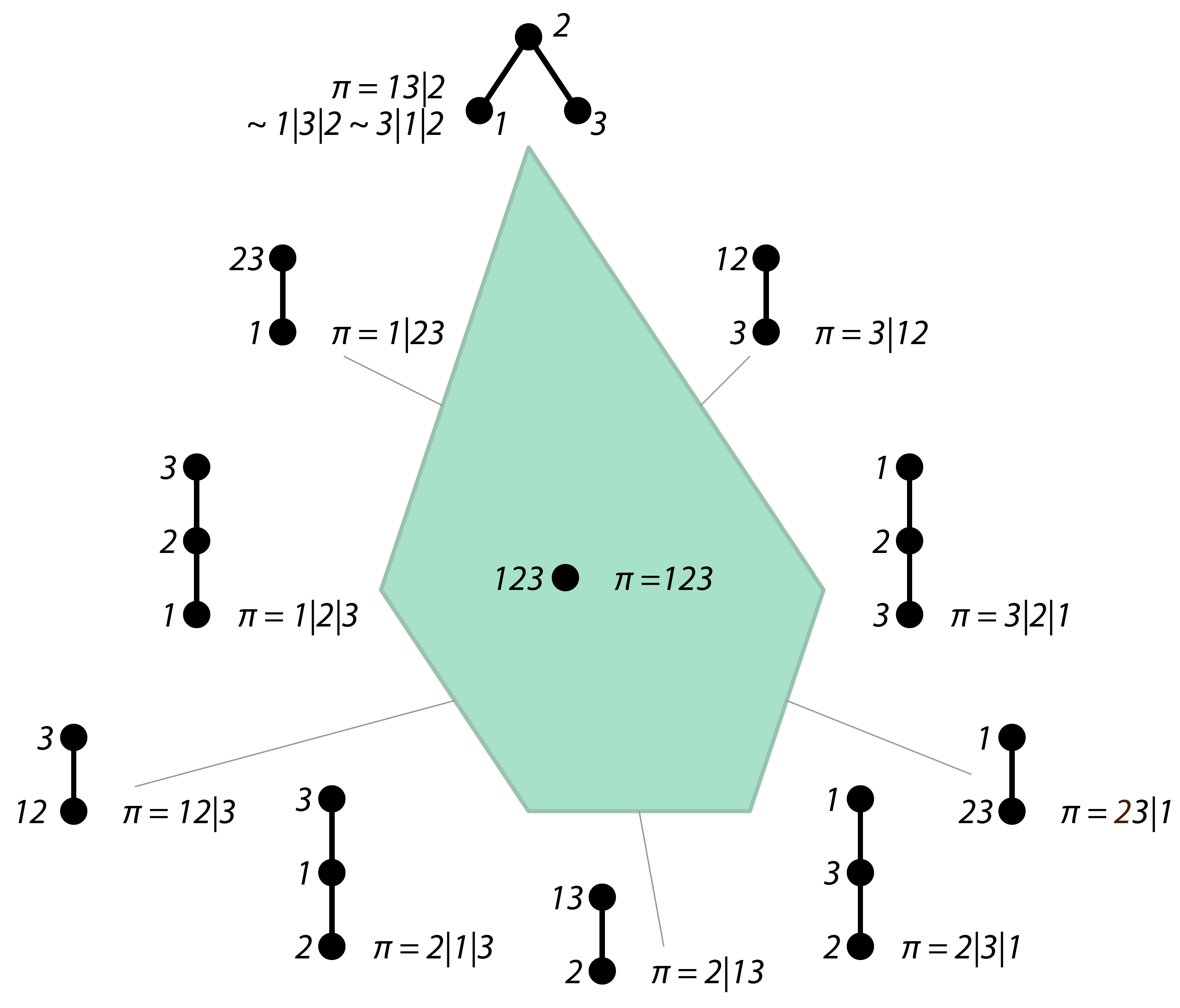}
\caption{The associahedron ${\mathcal K}_3$, whose faces are labeled by $\B(P_3)$-forests. Next to each face, we have also indicated the partitions of $[3]$ which maximize it.  
\label{fig:assoc}}
\end{figure}

\subsection{\textsf{Nestomultiplihedra and painted $\B$-forests.}}

\begin{deff}
A \emph{painted $\B$-forest} $\NN= (\N^-, \N^0, \N^+)$ is a $\B$-forest $\N$ together with a partition of the vertices into a downset $\N^-$, an antichain $\N^0$, and an upset $\N^+$ such that $\N^- \cup \N^0$ is a downset (and hence $\N^0 \cup \N^+$ is an upset).  The vertices of $\N^-, \N^0,$ and $\N^+$ are colored white, grey, and black, respectively.
\end{deff}

This can also be regarded as a definition of \emph{painted nested sets} for $\B$, since we are making no distinction between the nested sets for $\B$ and the $\B$-forests.

As a visual aid, we shade all half-edges above and below the black vertices, and above the grey vertices. The result is a connected ``coat of paint" starting at the root of each tree in the forest. 
Figure \ref{fig:paintedB-forest} shows a painted $\B$-forest for the building set of the graph in Figure \ref{fig:graph}(a). Here $\N^-=\{3,4,6,7\}, \, \N^0 = \{8,9\},$ and $\N^+ = \{15, 2\}$.

\begin{figure}[h]
\centering
\includegraphics[scale=.25]{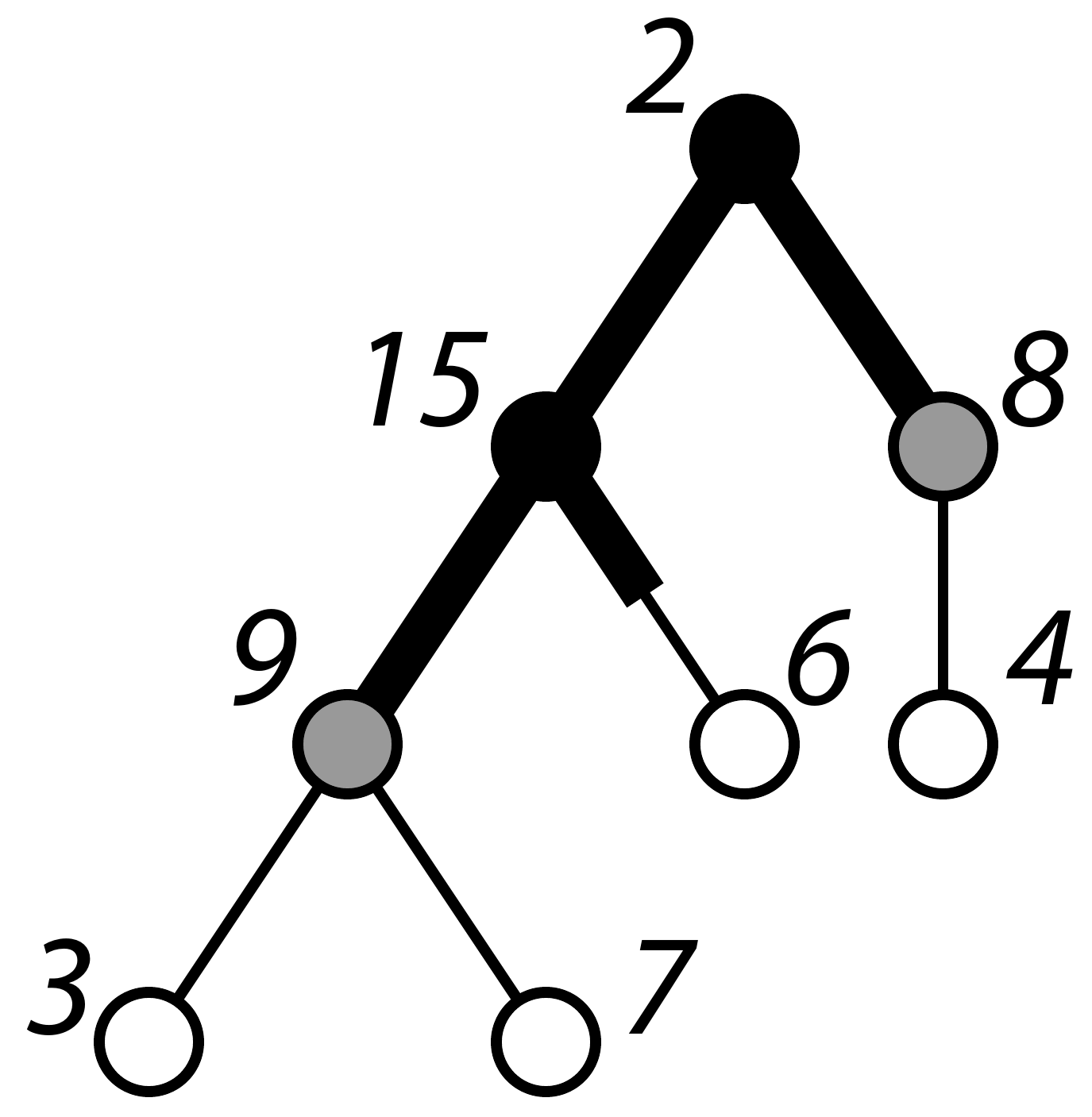}
\caption{A painted $\B$-forest. The vertices in $\N^-, \N^0,$ and $\N^+$ are shaded white, black, and grey, respectively.
\label{fig:paintedB-forest}}
\end{figure} 

This notion is compatible with the notion of painted trees in \cite{F}. 
When $\B(P_n) = \{[i,j] \, : \, 1 \leq i \leq j \leq n\}$ is the nested set of the path graph $P_n$,  the painted $\B(P_n)$-forests are in bijection with the \emph{painted trees} of \cite{F}. The bijection, which is illustrated in Figure \ref{fig:treeflip}, is as follows. Recall that a painted tree is planar and unlabeled. There are $n-1$ nooks between the pairs of adjacent siblings. Travel clockwise around the tree, starting at the root, and number the nooks $1, \ldots, n-1$ in the order that they are visited. Label each internal vertex with the set of numbers of its nooks. Also color each vertex white, grey, or black, according to whether its surroundings are completely uncolored, completely colored, or half colored. Finally remove the root and all the leaves, and turn the tree upside down. The result is a painted $\B(P_n)$-tree, and one easily checks that this procedure is reversible.

\begin{figure}[h]
\centering
\includegraphics[scale=.4]{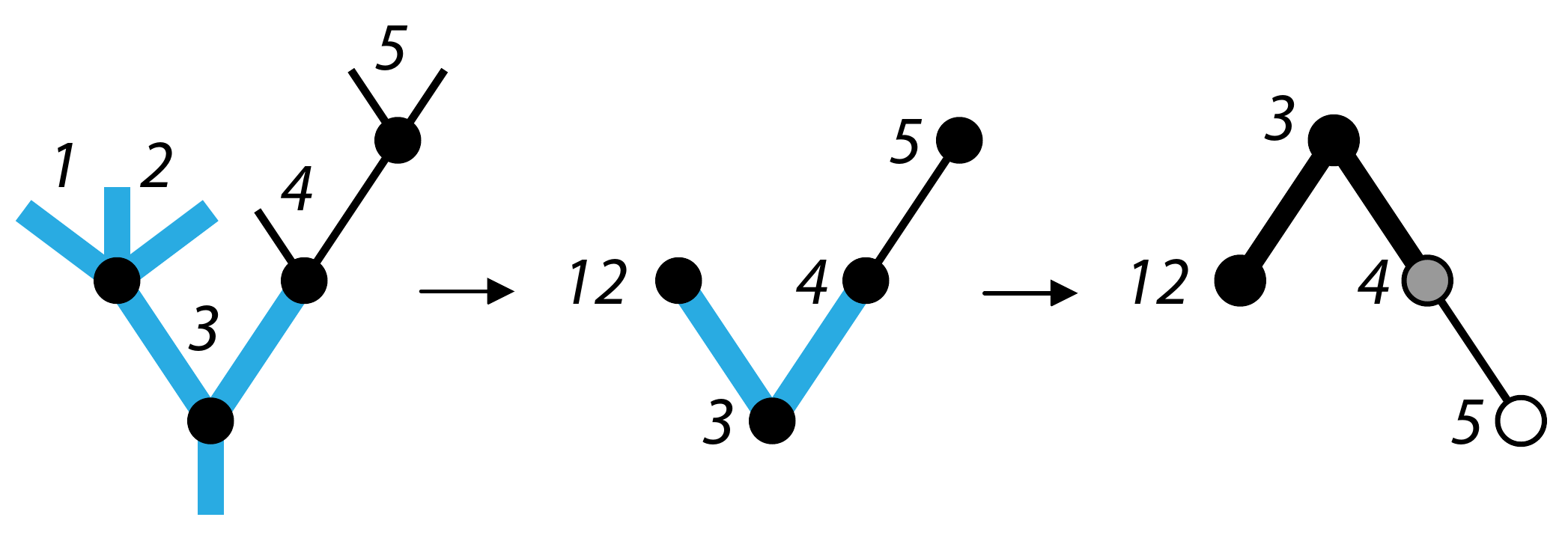}
\caption{A painted tree and the corresponding $\B(P_n)$-forest.
\label{fig:treeflip}}
\end{figure} 

Similarly, if $\B(G)$ is the building set of a graph $G$, then there is a natural bijection between the painted $\B(G)$-forests and the \emph{marked tubings} of \cite{DF}.

Given a painted $\B$-forest $\N$, the \emph{contraction} of an edge $ST$ is obtained by removing the edge $ST$ and relabeling the resulting merged vertex with the set $S \cup T$. If the vertices $S$ and $T$ had the same color, then the new vertex $S \cup T$ is given the same color. If they had different colors, then $S \cup T$ is colored grey.

When we contract an edge whose vertices are either both black (BB), both white (WW), or grey and white (GW), we obtain a painted $\B$-forest. When  we contract a BG edge $ST$, where $S$ is black and $T$ is grey, the result may not be a painted $\B$-forest. To obtain one, we need to contract {\bf all} BG edges $ST'$ where $T'$ is a grey descendent of $S$. We call this set of BG edges a \emph{BG} bunch. To contract a BW edge $ST$, we first need to contract the BG bunch hanging from $S$, if there is one; after that, we will be able to contract $ST$.

\begin{deff} Define a partial order on painted $\B$-forests by saying that $\N \geq \N'$ if the $\B$-forest $\N'$ is obtained from the $\B$-forest $\N$ by successively: 

\qquad $\bullet$ contracting a BB, WW, or GW edge, 

\qquad $\bullet$ contracting a BG bunch,

\qquad $\bullet$ converting a black vertex with only white successors into a grey vertex.

\qquad $\bullet$ converting a white vertex with a black predecessor into a grey vertex,

\end{deff}

\begin{figure}[h]
\centering
\includegraphics[scale=.4]{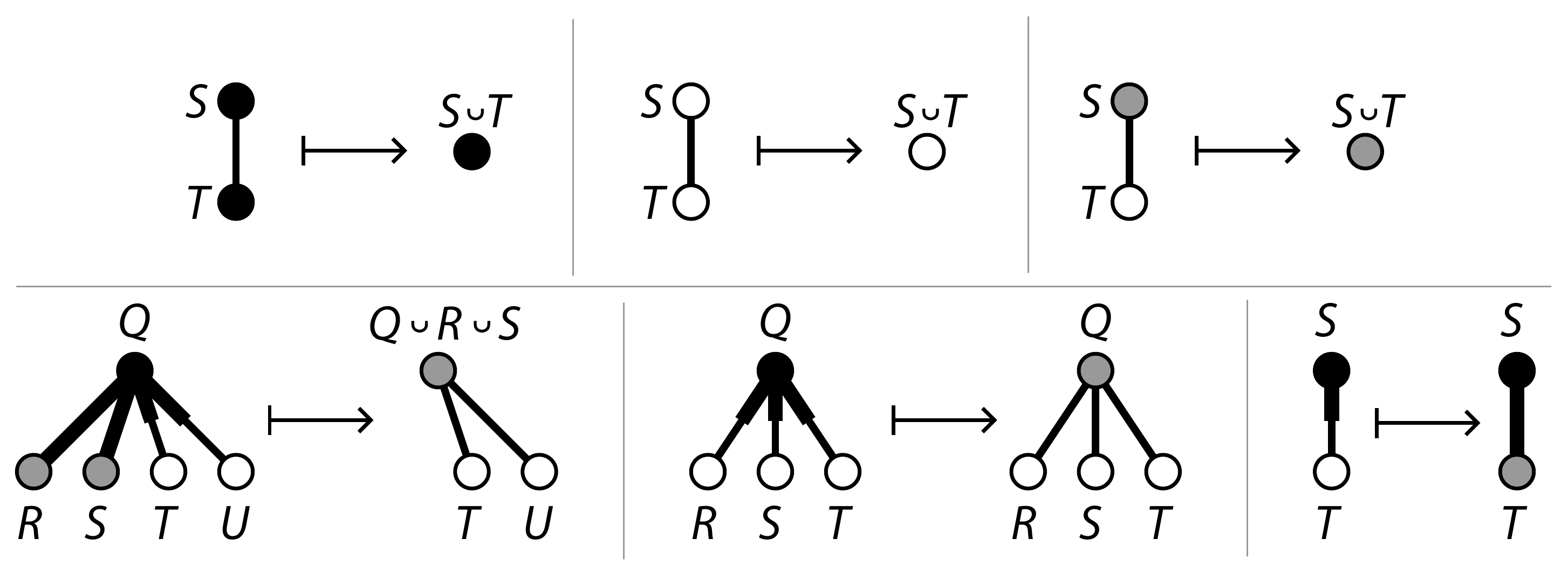}
\caption{The different ways of going down the poset of painted $\B$-forests. \label{fig:uptheposet}}
\end{figure} 

Figure \ref{fig:uptheposet} illustrates the six different operations that bring us down in the poset of painted $\B$-forests. Notice that ``contracting a BW edge (after contracting the corresponding BG bunch if there is one)" also brings us down in this poset, but such a contraction is a combination of the operations on the list. Therefore we do not include it.

%
%
%

Figure \ref{fig:J3} shows the multiplihedron ${\mathcal J}_3$ (which is also the graph multiplihedron ${\mathcal J}K_3$, as well as the nestomultiplihedron ${\mathcal J}\B(K_3)$ for the building set of $K_3$), whose faces are in order-preserving bijective correspondence with the painted trees on $[3]$. Our next theorem constructs the \emph{nestomultiplihedron}, which plays the analogous role for an arbitrary building set $\B$.

\begin{figure}[h]
\centering
\includegraphics[scale=.4]{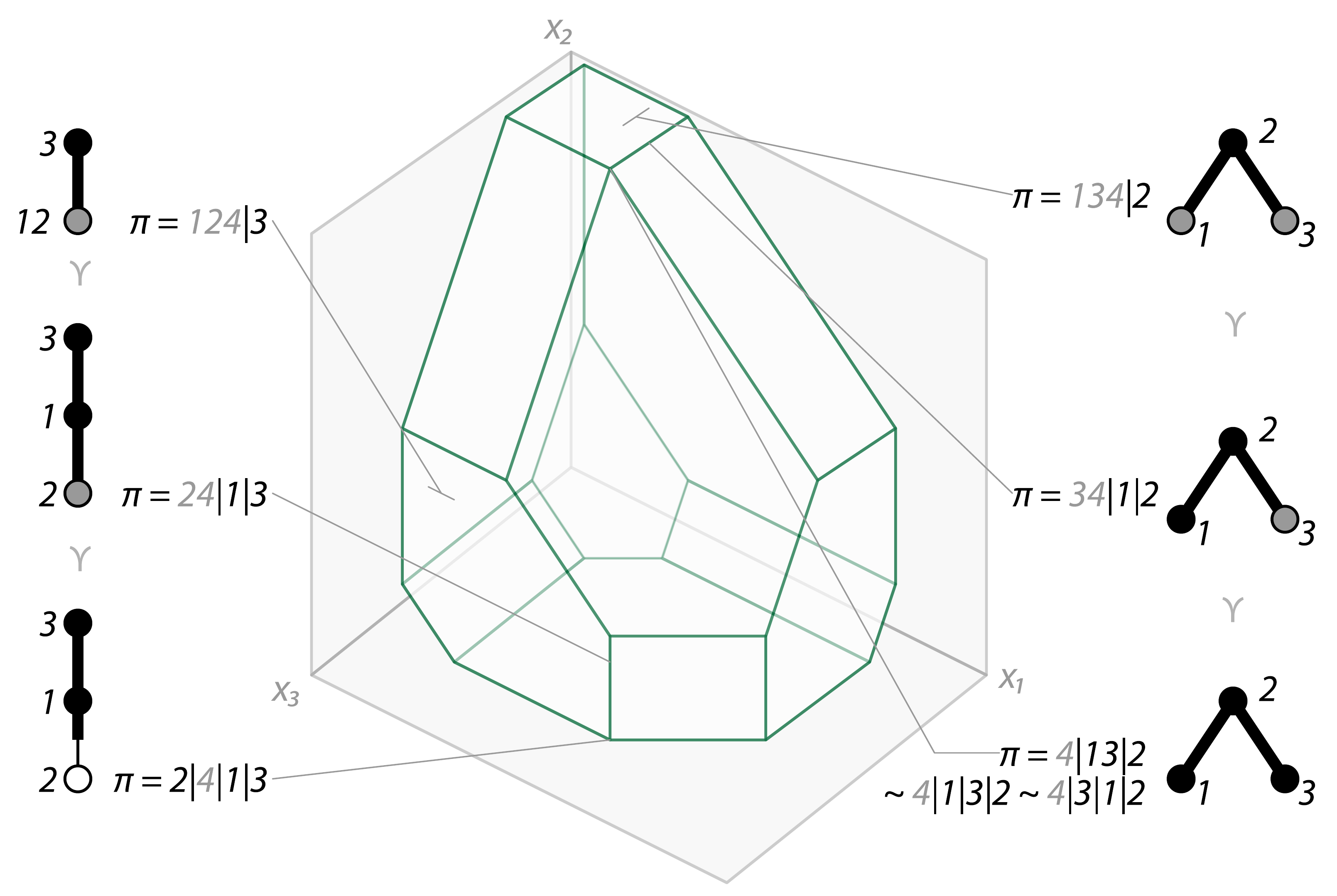}
\caption{The multiplihedron ${\mathcal J}_3 = {\mathcal J}K_3$ (projected onto the hyperplane $x_4=0$), whose faces are labeled by painted $\B(K_3)$-forests. Next to some of the faces, we indicate the corresponding $\B(K_3)$-forest, as well as the partitions of $[4]$ which maximize them. \label{fig:J3}}
\end{figure}

\begin{theorem} \label{th:nesto}
The face poset of the \emph{nestomultiplihedron} 
\[
{\mathcal J}{\B}:= \sum_{B \in \B} \Delta_B + \sum_{B \in \B} \Delta_{B \cup \{n+1\}}
\]
is isomorphic to the opposite of the poset of painted $\B$-forests.
\end{theorem}

\begin{proof}
Let $\pi'$ be an ordered partition of $[n+1]$ and $\pi$ the partition of $[n]$ obtained by removing $n+1$ from $\pi'$. Consider the face $({\mathcal J}{\B})_{\pi'}$ maximized in direction $\pi'$:
\begin{equation}\label{eq:face}
({\mathcal J}{\B})_{\pi'} = \sum_{B \in \B} (\Delta_B)_{\pi} + \sum_{B \in \B} (\Delta_{B \cup \{n+1\}})_{\pi'}.
\end{equation}
For each $B \in \B$ let $j(B)$ be the largest $j$ for which $B$ intersects $\pi_j$. Notice that $j$ is a weakly increasing function, in the sense that $N \subset M$ implies $j(N) \leq j(M)$. Write $B_{\pi} := B \cap \pi_{j(B)}$, so $(\Delta_B)_{\pi} = \Delta_{B_\pi}$. Let
\[
\N = \N_\pi(\B) := \{N \in \B \, : \, j(N) < j(M) \textrm{ for all } M \in \B \textrm{ with } N \subsetneq M\}.
\]
Alternatively, construct $\N$ recursively by the following branching procedure: The maximal elements $N_1, \ldots, N_k$ of $\B_{\max}$ are in $\N$, and every other $B \in \B$ is a subset of one such $N_i \in \B_{\max}$. If $B \cap (N_i)_\pi \neq \emptyset$ then $B$ is not in $\N$. The remaining $B \subset N_i \backslash (N_i)_\pi$ are the elements of the induced building set $\B|_{ N_i \backslash (N_i)_\pi}$. Construct the corresponding nested set $\N_i$ in each $\B|_{ N_i \backslash (N_i)_\pi}$, and let $\N = \B_{\max} \cup \N_1 \cup \cdots \cup \N_k$. The result is a nested set.

For the building set $\B$ of the graph in Figure \ref{fig:graph}(a), and the ordered partition $\pi = 347|6|89|15|2$, we obtain the nested set $\N = \{3,4,6,7,379,48,135679, 123456789\}$ of Figure \ref{fig:graph}(b). Note that Figure \ref{fig:B-forest} encodes the branching procedure described in the previous paragraph.

If $n+1$ was added between blocks $\pi_i$ and $\pi_{i+1}$ of $\pi$ to be in its own block in $\pi'$, let $k=i+\frac12$.  Otherwise, if $n+1$ was added to block $\pi_i$, let $k=i$. Let 
\begin{eqnarray*}
\N^+ &=& \{N \in \N \, : \, j(N) > k\}, \\
\N^0 &=& \{N \in \N \, : \, j(N) = k\}, \\
\N^- &=& \{N \in \N \, : \, j(N) < k\}
\end{eqnarray*}
By the definition of $\N$, the set $\N^0$ is an antichain. Also, since $j(\cdot)$ is weakly increasing, $\N^-$ and $\N^- \cup \N^0$ are order ideals and $\N^+$ is an order filter. Therefore $\NN_{\pi'}(\B) := \NN=(\N^+, \N^0, \N^-)$ is a painted $\B$-forest.

In the previous example, if $\pi' = 347|6|89{\bf 10}|15|2$ we obtain the painted $\B$-forest $\NN$ of Figure \ref{fig:paintedB-forest}.

We plan to label the face $({\mathcal J}{\B})_{\pi'}$ with the painted $\B$-forest $\NN$. In order to do that, we need to show that $\NN$ actually determines $({\mathcal J}{\B})_{\pi'}$. By (\ref{eq:face}) it suffices to show that $\NN$ determines $B_\pi$ and $(B \cup \{n+1\})_{\pi'}$ for all $B \in \B$. 
One easily checks (see \cite{Po}) that if $N \in \N$ then $N_\pi =
N - \bigcup_{M \in \N \, : \, M \subsetneq N} M$, which depends only on $\N$. Now for an arbitrary $B \in \B$ let $N$ be the minimal set in $\N$ containing $B$. From the expression for $N_\pi$ above we see that $N_\pi \cap B$ is non-empty, and therefore $B_\pi = N_\pi \cap B = N \cap \pi_{j(N)} \cap B$, which only depends on $\N$. Finally observe that  $(B \cup \{n+1\})_{\pi'}$ equals $B_\pi$ if $N \in \N^-$, or $B_\pi \cup \{n+1\}$ if $N \in \N^0$, or $\{n+1\}$ if $N \in \N^+$, and so it is determined by $\NN$.

\medskip

Having shown that every face is labeled by a painted $\B$-forest, we need to show that every painted forest $\NN=(\N^+, \N^0, \N^-)$ labels a face. Label the nodes of $\NN$ using all the numbers $1, 2, \ldots, m$, possibly with repetitions, strictly increasingly up the forest, in such a way that the nodes in $\N^-$ get labels $1,\ldots,k-1$, the nodes in $\N^0$ all get the label $k$ (if $\N^0 \neq \emptyset$), and the nodes in $\N^+$ get the labels $k+1, \ldots, m$. 
Give $n+1$ the label $k$. In general there are many such labellings. Now consider the partition $\pi'$ of $[n+1]$ which places the nodes labeled $i$ in part $\pi'_i$. 
We claim that the face $(P_{\B})_{\pi'}$ is labeled by the painted $\B$-forest $\NN$. 

\begin{figure}[h]
\centering
\includegraphics[scale=.3]{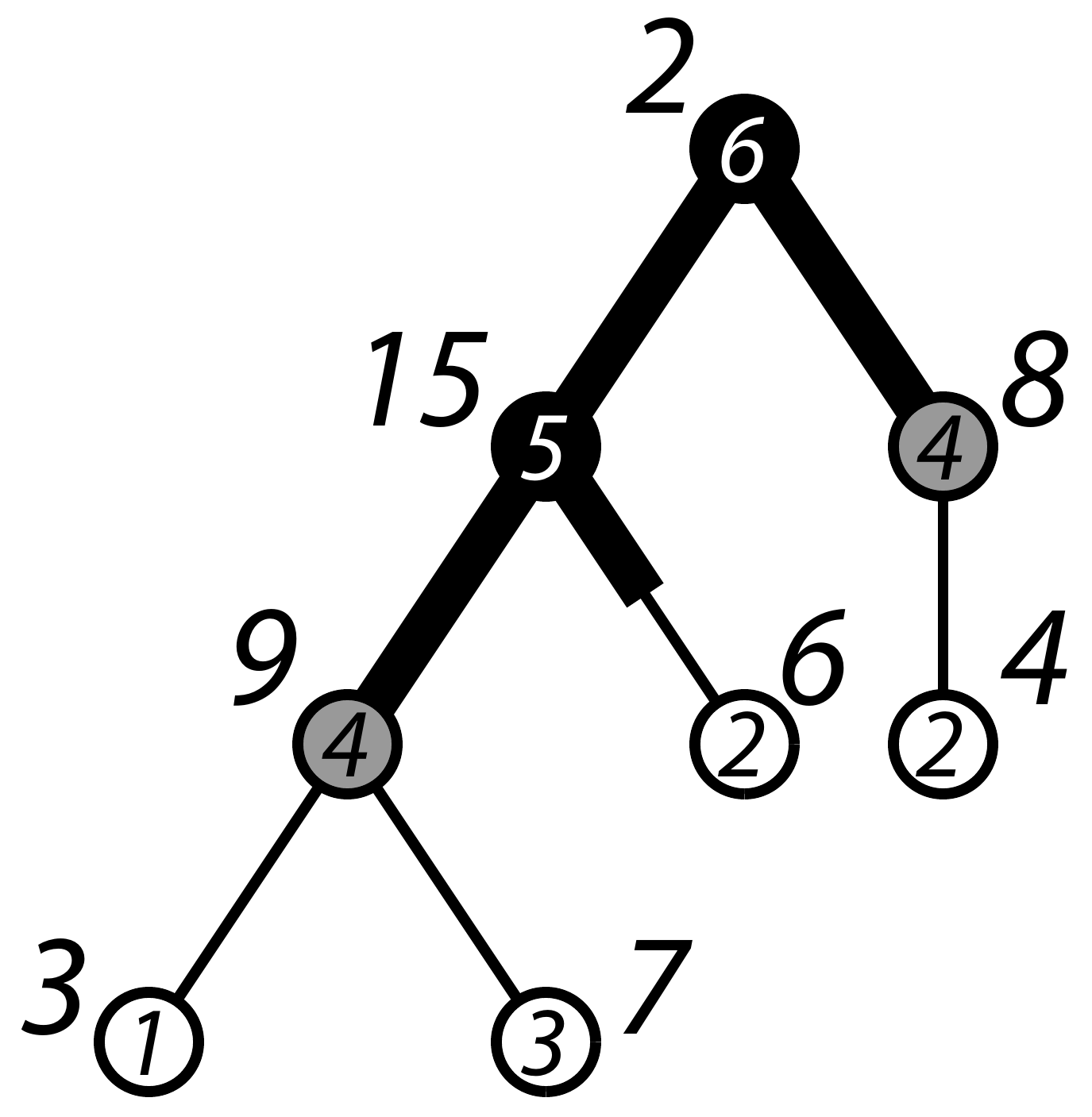}
\caption{A painted $\B$-forest and a suitable labeling of its nodes, which gives $\pi'=3|46|7|89{\bf 10}|15|2$. \label{fig:labeling}}
\end{figure}

First we show that $\N \subseteq \N_{\pi'}(\B)$. As before, let $j(B) = \max\{j \, : \, B \cap \pi'_j \neq \emptyset\}$. Indeed, if we had $N \in \N \backslash \N_{\pi'}(\B)$, we would have $N \subsetneq B \in \B$ with $j(N) = j(B) = j$. Consider the maximal sets $N_1, \ldots, N_k$ of $\N$ that $B$ intersects. They must all have $j(N_i) \leq j$. Since the numbers on the nodes increase strictly up the forest, $N$ must be one of the $N_i$s, and it cannot be the only one. By property (B1) of building sets we conclude that $B \cup N_1 \cup \cdots N_k = N \cup N_1 \cup \cdots N_k \in \B$, which contradicts property (N2) of nested sets. 

Now we show that $\N_{\pi'}(\B) \subseteq \N$. Assume we had  $B \in \N_{\pi'}(\B) \backslash  \N$. Consider the minimal $N \in \N$ containing $B$. Since $j(B) < j(N)$, $B$ cannot intersect $N_{\pi'}$. Let $N_1, \ldots, N_k$ be the maximal sets in $\N$ that $B$ intersects. They are all strict subsets of $N$, and there are at least two of them by the minimality of $N$. Then by (B1) we have $B \cup N_1 \cup \cdots N_k = N_1 \cup \cdots N_k \in \B$, which again contradicts (N2).

We conclude that $\N = \N_{\pi'}(\B)$. From the construction of $\pi'$ we see that block $k$ of $\pi$ consists of $n+1$ and the union of the sets in $\N^0$, so we also have  $\NN = \NN_{\pi'}(\B)$ as desired.

Finally, we check that this bijection between faces of ${\mathcal J}{\B}$ and painted $\B$-forests is order-reversing. Let $F_1$ be a face given by a painted $\B$-forest $\N_1$ and let $\pi^1$ be a finest partition of $[n+1]$ realizing it, so $F_1 =P_{\pi^1}$. Consider a face $F_2$ covering $F_1$; say it corresponds to tree $\N_2$. We can write $F_2=P_{\pi^2}$ for a partition $\pi^2$  obtained from $\pi^1$ by merging two parts. If both parts precede (or both succeed) $n+1$ in $\pi^1$, then we are contracting a WW edge (or a BB edge) to get from $\N_1$ to $\N_2$. If $n+1$ is its own block in $\pi^1$, and it is being merged with a block preceding it (or succeeding it), then we are turning one or more white (or black) vertices into grey vertices. If $n+1$ is not its own block, and it is being merged with a block preceding it (or succeeding it), then we are contracting a GW edge (or contracting a BG bunch.) Therefore $\N_1$ covers $\N_2$. The converse follows by a similar and easier argument.
\end{proof}

Figure \ref{fig:labeling} shows that the \emph{painted nested set complex}, which is dual to the nestomultiplihedron, is not necessarily a simplicial complex.

\begin{cor}
The lifting of the nestohedron $\mathcal{K}\B$ is the nestomultiplihedron $\mathcal{J}{\B}$.
\end{cor}

\begin{proof}
In light of Proposition \ref{prop:q-lift} and Theorem \ref{th:nesto}, this follows from the fact that the polytopes $\sum_{B \in \B} \Delta_B + \sum_{B \in \B} \Delta_{B \cup \{n+1\}}$ and $q\sum_{B \in \B} \Delta_B + (1-q)\sum_{B \in \B} \Delta_{B \cup \{n+1\}}$ have the same combinatorial type. \end{proof}

\begin{rem}
In \cite{DF}, Devadoss and Forcey asked for a nice Minkowski decomposition of the \emph{graph multiplihedron} ${\mathcal K}G$. By definition, ${\mathcal K}G$ is combinatorially isomorphic to the nestomultiplihedron for the building set $\B(G)$ of the graph $G$. Therefore Theorem \ref{th:nesto} offers an answer to their question.
\end{rem}

\begin{rem}\label{rem:nest}
Notice that the nestohedron is, up to translation (\emph{resp.} scaling) the face of the nestomultiplihedron that maximizes (\emph{resp.} minimizes) the linear function $x_{n+1}$. Therefore the face poset of the nestomultiplihedron ${\cal J} B$ contains two copies of the face poset of the nestohedron ${\cal N} B$, corresponding to the subposets of fully painted and fully unpainted $\B$-forests, respectively. This gives another proof of Theorem \ref{th:nest}.
\end{rem}

\section{\textsf{$\pi$-liftings and volumes.}}\label{section: face q-liftings}

We will now modify the lifting operation and define, for each ordered partition $\pi$ of $[n]$ and $0 \leq q \leq 1$, the $(\pi, q)$-lifting $P^\pi(q)$.
 This construction is useful in that it subdivides the polytope $P(q)$ into pieces whose volumes are easier to compute; \emph{i.e.}
\[
P(q) = \bigcup_{\pi\in\mathcal{P}^n}P^\pi(q), \qquad \textrm{int } P^{\pi}(q)  \cap \textrm{int } P^{\pi'}(q) = \emptyset \,\, \textrm{ for } \pi \neq \pi',
\]
so
\[
\vol_n(P(q)) = \sum_{\pi\in \mathcal{P}^n} \vol_n(P^\pi(q)).
\]
We will see that $\vol_n(P^\pi(q))$ is an interesting polynomial in $q$, which we will explore in greater depth in Part 2. 

For the sake of visualization and the cleanliness of formulas, in this section we will treat $P(q)$ as a full-dimensional polytope in $\mathbb{R}^n$ via projection onto the hyperplane $x_{n+1}=0$, rather than as a polytope of codimension $1$ in $\mathbb{R}^{n+1}$. Thus if $P = P_n(\{z_I\})$ then it follows from Definition \ref{prop: z_I} that $P(q)$ will have hyperplane description
\[
P(q) = \left\{ x\in \mathbb{R}^n :  q z_I \leq \sum_{i\in I} x_i \leq z_{[n]}-z_{[n]\setminus I} \text{ for all } I\subseteq [n]\right\}.
\]

\begin{deff}
Let $P$ be a generalized permutahedron in $\R^n$. 
Let $\pi=\pi_1|\cdots|\pi_k$ be an ordered partition of $[n]$ and let $0 \leq q \leq 1$. Let $P_\pi$ be the face of $P$ that maximizes a linear functional of type $\pi$. For $i=0,\dots,k$ construct a modified copy $P_\pi^i$ of $P_\pi$ by applying a factor of $q$ to the coordinates of the vertices of $P_\pi$ whose indices belong to the first $i$ blocks of $\pi$,  $\pi_1\cup\cdots\cup \pi_i$. The convex hull of all of these modified copies of $P_\pi$ is the \emph{$(\pi,q)$-lifting} of $P$. We denote it as $P^\pi(q)$, and sometimes we simply call it the \emph{$\pi$-lifting} of $P$.
\end{deff}

Note that each ordered partition $\pi$ corresponds to a different $\pi$-lifting $P^\pi(q)$. Even if $P_\pi = P_\mu$, the $\pi$-liftings $P^\pi(q)$ and $P^\mu(q)$ will be distinct for $\pi\neq\mu$.

\begin{exa}
Consider the associahedron $\mathcal{K}(4)$. Since $\mathcal{K}(4)_{1|3|2}$ is the point $(1,4,1)$,  the $\pi$-lifting
\[
\mathcal{K}(4)^{1|3|2}(q) = \text{conv}\{ (1,4,1), (q,4,1), (q,4,q), (q,4q,q) \}.
\] 

\begin{figure}[h]
\centering
\includegraphics[scale=.24]{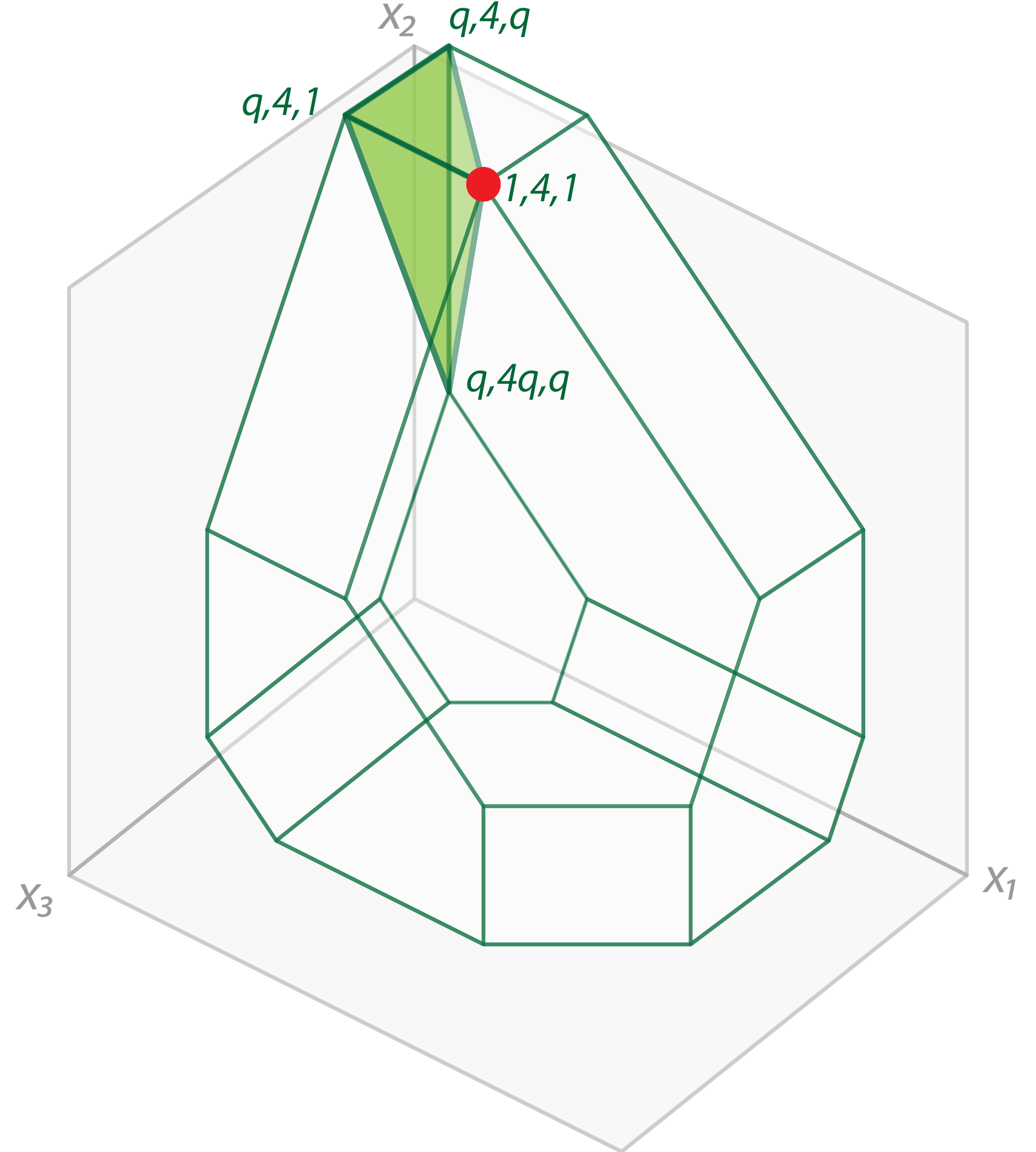} \includegraphics[scale=.24]{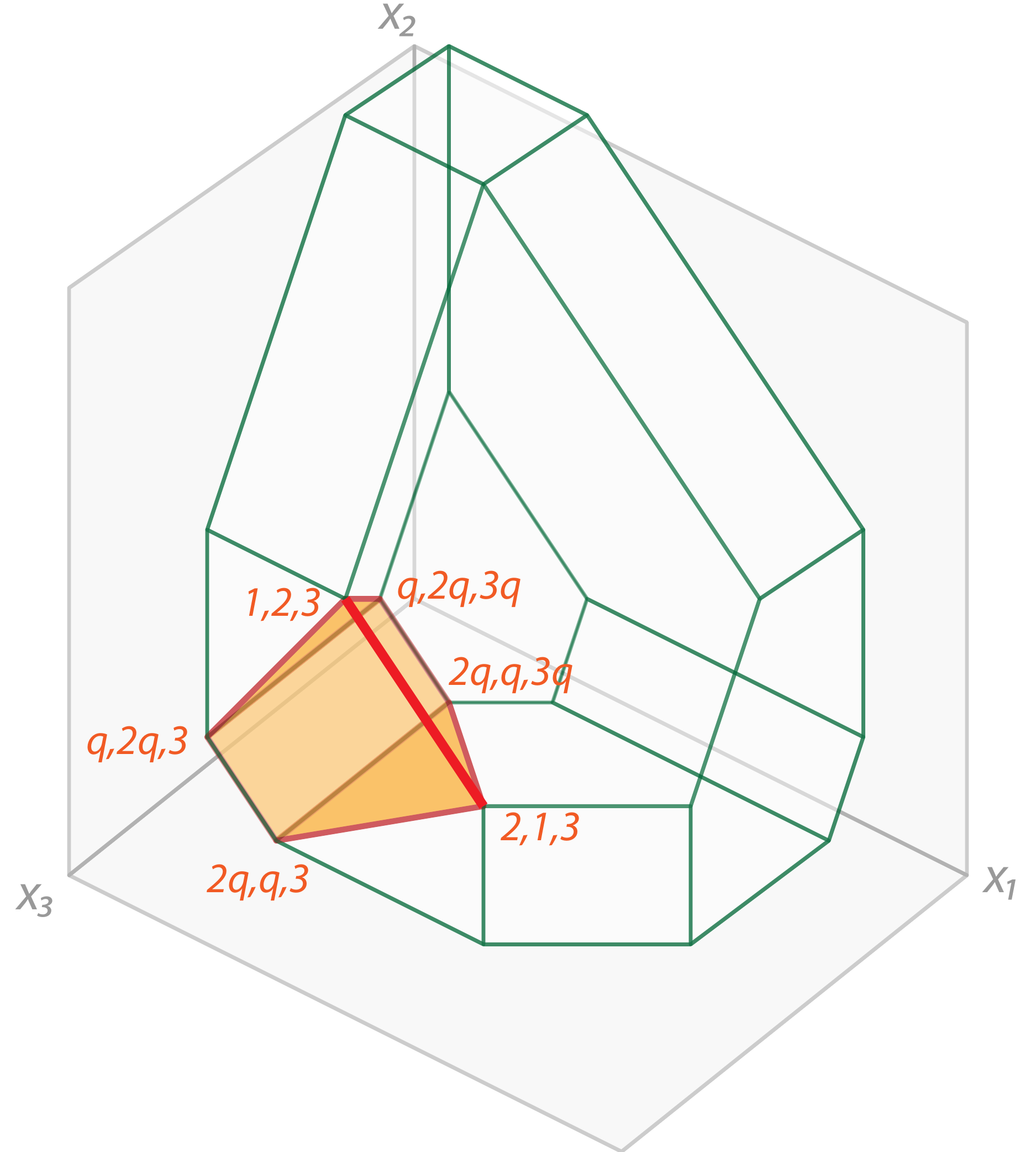} \includegraphics[scale=.24]{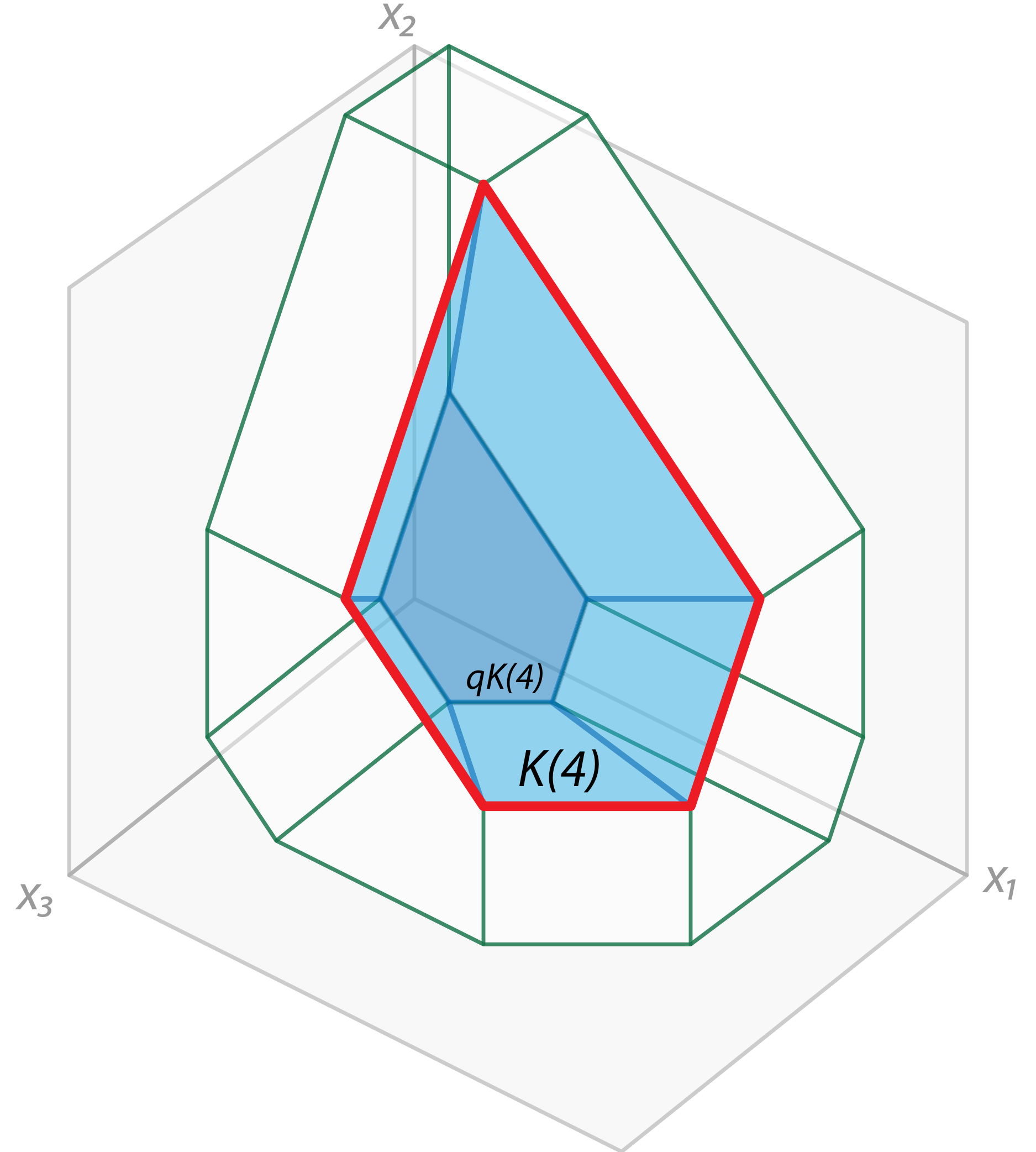} 
\caption{Three $\pi$-liftings of the associahedron $\mathcal{K}(4)$: $\mathcal{K}(4)^{1|3|2}(q)$, $\mathcal{K}(4)^{12|3}(q)$, and $\mathcal{K}(4)^{123}(q)$. The bold regions represent the faces $\mathcal{K}(4)_\pi$.}\label{fig: face liftings}
\end{figure} 
\end{exa}
%
%

\begin{deff}
For a subset $I\subseteq [n]$ define $x_I := \sum_{i\in I} x_i$. For a generalized permutahedron $P=P_n(\{z_I\})$ and an ordered partition $\pi = \pi_1|\cdots|\pi_k$ define
\[
z_\pi^{\pi_i} := z_{\pi_1\cup\cdots\cup \pi_i} - z_{\pi_1\cup\cdots\cup \pi_{i-1}}
\]
For a minimal refinement $\pi'=\pi_1|\cdots|\pi_{i-1}|C_i|D_i|\pi_{i+1}|\cdots|\pi_k$, where $\pi_i=C_i \sqcup D_i$ is a disjoint union, we have
\[
z_{\pi'}^{C_i} := z_{\pi_1\cup\cdots\cup \pi_{i-1}\cup C_i} - z_{\pi_1\cup\cdots\cup \pi_{i-1}} \text{ and}
\]
\[
z_{\pi'}^{D_i} := z_{\pi_1\cup\cdots\cup \pi_{i}} - z_{\pi_1\cup\cdots\cup \pi_{i-1}\cup C_i}.
\]
\end{deff}

\begin{prop}\label{prop: wedge ineqs}
For a generalized permutahedron $P=P_n(\{z_I\})$ and an ordered partition $\pi=\pi_1|\cdots|\pi_k$ the $\pi$-lifting $P^\pi(q)$ has the hyperplane description:
\begin{align}
& q \,\, \leq \,\, \frac{x_{\pi_1}}{z_\pi^{\pi_1}}\,\,  \leq \,\, \cdots\,\,  \leq\,\,  \frac{x_{\pi_k}}{z_\pi^{\pi_k}}\,\,  \leq\,\,  1, \label{simplicial} \tag{S} \\
& \frac{x_{C_i}}{z_{\pi'}^{C_i}} \,\, \geq \,\,  \frac{x_{D_i}}{z_{\pi'}^{D_i}}  \quad\textrm{ for all $i$ and all disjoint decompositions } \pi_i = C_i \sqcup D_i. \tag{F} \label{facial}
\end{align}
\end{prop}
For reasons to become clear later, we call the inequalities of the first type the \emph{simplicial inequalities} of $P^\pi(q)$, and those of the second type the \emph{facial inequalities}. Since $x_{C_i}+x_{D_i} = x_{\pi_i}$ and $z_{\pi'}^{C_i}+z_{\pi'}^{D_i}=z_{\pi}^{\pi_i}$, the facial inequalities can be rewritten as 
\[
\frac{x_{C_i}}{z_{\pi'}^{C_i}} \geq  \frac{x_{\pi_i}}{z_\pi^{\pi_i}} \quad \textrm{ or equivalently as } \quad 
\frac{x_{D_i}}{z_{\pi'}^{D_i}} \leq  \frac{x_{\pi_i}}{z_\pi^{\pi_i}}.
\]

\begin{proof}[Proof of Proposition \ref{prop: wedge ineqs}]
First we claim that any vertex (and hence any point) of $P^\pi(q)$ satisfies the given inequalities. The face $P_\pi$ consists of the points $x$ in $P$ that satisfy $x_{\pi_i} = z_\pi^{\pi_i}$ for $i=1,\dots,k$. For any vertex $v$ of $P_\pi^j$, $v_{\pi_i}/z_\pi^{\pi_i}$ equals $q$ if $i \leq j$ and $1$ if $i>j$, so $v$ satisfies the simplicial inequalities (\ref{simplicial}). Now, for any vertex $v$ of $P_\pi$ we have 
$v_{D_i} + (z_{[n]} - z_{\pi_1 \cup \cdots \cup \pi_i}) = v_{D_i}+v_{\pi_{i+1}}+ \cdots + v_{\pi_k} = v_{D_i \cup \pi_{i+1} \cup \cdots \cup \pi_k} \leq z_{[n]}-z_{\pi_1 \cup \cdots \cup \pi_{i-1} \cup C_i}$ so 
${v_{D_i}}/{z_{\pi'}^{D_i}} \leq 1 = {v_{\pi_i}}/{z_\pi^{\pi_i}}.$ 
So all vertices of $P_\pi$, and therefore those of 
$P_\pi^j$ 
satisfy the facial inequalities (\ref{facial}) as well.
The claim follows.

Conversely, given a point $x$ which satisfies the given inequalities, we show that $x \in P^\pi(q)$. For a subset $S \subseteq [n]$ and a vector $x \in \R^n$, write $x|_S$ for the ``restriction" vector in $\R^S$ whose coordinates are the $S$-coordinates of $x$.
Define $p \in \R^n$ by
\[
p|_{\pi_i} = x|_{\pi_i} \cdot \frac{z_\pi^{\pi_i}}{x_{\pi_i}} \qquad \textrm{ for } i=1, \ldots, k.
\]
We have $p_{\pi_i} = z_\pi^{\pi_i}$ for all $i$. Let $p^i$ be the point obtained from $p$ by multiplying the entries in $\pi_1 \cup \cdots \cup \pi_i$ by $q$. We will show that $x$ is a convex combination of $p^0, p^1, \ldots, p^k$, and that these $k+1$ points are in $P^\pi(q)$. This will imply that $x \in P^\pi(q)$.

To show the first claim, we write  $a_i = \left(\frac{x_{\pi_i}}{z_\pi^{\pi_i}}\right)/(1-q)$, and compute
\begin{eqnarray*}
x &=& \sum_{i=1}^k x|_{\pi_i}  = \sum_{i=1}^k p|_{\pi_i} \cdot \frac{x_{\pi_i}}{z_\pi^{\pi_i}} = \sum_{i=1}^k \frac{p^{i-1}-p^{i}}{1-q}   \frac{x_{\pi_i}}{z_\pi^{\pi_i}}
\\ &=& \sum_{i=1}^k (p^{i-1}-p^{i}) \cdot a_i  
= p^0 a_1 + \sum_{i=1}^{k-1} p^i(a_{i+1} - a_i) - p^k a_k \\
&=& p^0 \left(a_1-\frac{q}{1-q}\right) + \sum_{i=1}^{k-1} p^i\left(a_{i+1} - a_i\right) +p^k\left(\frac{1}{1-q}-a_k\right),
\end{eqnarray*}
where the coefficients are non-negative by assumption, and add up to $1$, as desired.

Now we prove that $p \in P_\pi$, which will imply that $p^i \in P_\pi^i \subset P^\pi(q)$ for all $i$. By definition $p$ satisfies all the equalities $x_{\pi_i} = z_\pi^{\pi_i}$ for $i=1,\dots,k$ that hold in the face $P_\pi$. Now let us check that it satisfies all inequalities as well. We need to check that $p_C \geq z_C$ for all $C \subseteq [n]$. Write $C = C_1 \cup \cdots \cup C_k$ where $C_i:= C \cap \pi_i \subseteq \pi_i$, so $p_C = p_{C_1} + \cdots + p_{C_k}$. 
Applying the facial inequalities, we have
\[
p_{C_i} = x_{C_i} \cdot \frac{z_\pi^{\pi_i}}{x_{\pi_i}} \geq z_{\pi'}^{C_i} = z_{\pi_1 \cup \cdots \cup \pi_{i-1} \cup C_i} - z_{\pi_1 \cup \cdots \cup \pi_{i-1}} 
\]

The supermodularity of $z$ then gives
\[
p_{C_i} \geq z_{C_1 \cup \cdots \cup C_{i-1} \cup C_i} \,\, - \,\, z_{C_1 \cup \cdots \cup C_{i-1}} 
\]
which implies that $p_C \geq z_{C_1 \cup \cdots \cup C_k} = z_C$ as desired.
\end{proof}


\begin{cor}\label{prop: face decomp}
The $\pi$-lifting $P^\pi(q)$ can be decomposed into the Minkowski sum
\[
P^\pi(q) = qP_\pi + (1-q) P^\pi(0).
\]
\end{cor}
\begin{proof}
Proposition \ref{prop: wedge ineqs} tells us the facet directions of the three polytopes involved. The result then follows from the fact that hyperplane parameters are additive under Minkowski sums. 
\end{proof}
%

Now we show that the different $\pi$-liftings $P^\pi(q)$ fit together to subdivide $P(q)$, as illustrated in Figure \ref{fig:exploded}.

\begin{figure}[h]
\centering
\includegraphics[scale=.45]{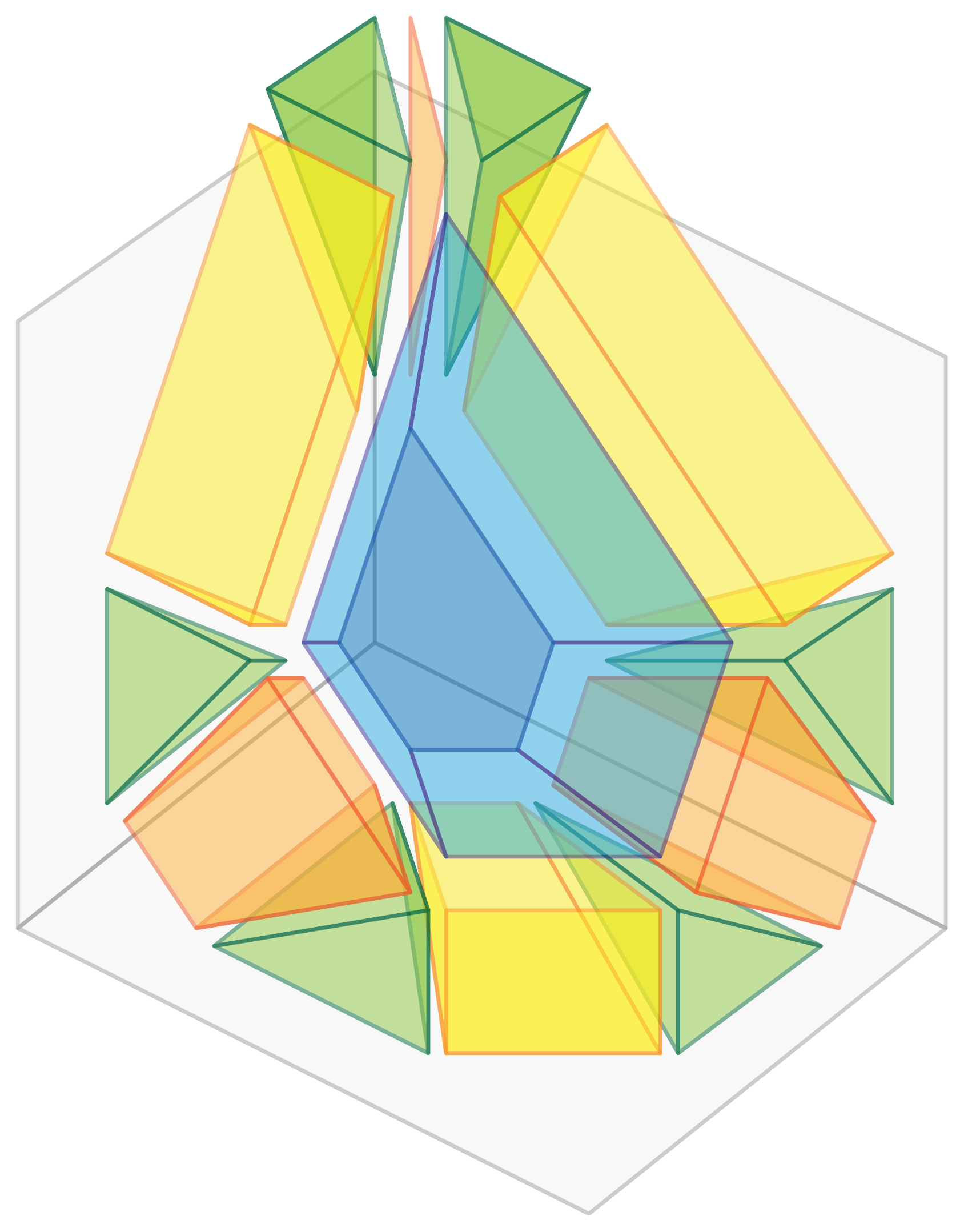} 
\caption{The subdivision of a lifted generalized permutahedron.}\label{fig:exploded}
\end{figure}

\begin{prop}
The set of $\pi$-liftings $\{P^\pi(q):\pi \text{ an ordered partition of }[n]\}$ forms a subdivision of the $q$-lifted polytope $P(q)$.
\end{prop}
\begin{proof}
Let $\pi=\pi_1|\cdots|\pi_k$ be an ordered partition and let $A_i = \pi_1\cup\cdots\cup \pi_i$. Recall that we have assumed that $P$ has been translated to sit in the interior of the positive orthant of $\mathbb{R}^n$. This means that every $x\in P$ will have all strictly positive coordinates, and $z_I< z_J$ for $I\subsetneq J$. We will now reinterpret the inequality description parameters of $P^\pi(q)$ in terms of slopes. For a point $x\in \mathbb{R}^n$ let $v_{I} = (z_{I},x_I)\in \mathbb{R}^2$, where $x_I = \sum_{i\in I} x_i$ as above. For $x\in P^\pi(q)$ the term $\frac{x_{\pi_{i}}}{z^{\pi_{i}}} = \frac{x_{A_i}-x_{A_{i-1}}}{z_{A_i}-z_{A_{i-1}}}$ is the slope of the segment joining $v_{A_{i-1}}$ and $v_{A_i}$. Thus the simplicial inequalities in Proposition \ref{prop: wedge ineqs} can be interpreted as stating that, starting at the origin $v_{A_0}=v_\emptyset$, the points $v_{A_0},v_{A_1}, v_{A_2}, \dots, v_{A_k}$ form a broken line of ascending slopes. Similarly, the facial inequalities state that all points $v_C$ with $A_{i-1} \subset C \subset A_i$ lie on or above the segment connecting $v_{A_{i-1}}$ and $v_{A_i}$.

Now given a point $x\in P(q)$ construct a partition $\pi$ as follows. Draw the $2^n$ points $v_I$, take the convex hull to create a polygon $Q$, and look at the ``lower hull" of $Q$, which consists of the edges $Q$ that maximize a linear functional whose second component is nonpositive. This will form a broken line of ascending slopes connecting vertices $v_{A_0}, v_{A_1},\dots,v_{A_k}$. Because the $x_i$ are strictly positive we know $v_{A_0}$ will be the origin, and because of the increasing condition on the $z_I$ we know $A_k = [n]$. Now we claim that $A_{i-1}\subset A_i$ for all $i$. 

Suppose by way of contradiction that, ordered from left to right, $v_A$ and $v_B$ are consecutive vertices in the lower hull of $Q$, but that $A\not\subset B$. By the increasing condition on the $z_I$ we have $z_{A\cap B} < z_A < z_B < z_{A\cup B}$. Moreover, because $v_A$ and $v_B$ are vertices of the lower hull of $Q$ we know that the slope of the line segment connecting $v_{A\cap B}$ and $v_A$ is strictly less than the slope of the segment between $v_A$ and $v_B$, which is in turn strictly less than the slope of the segment between $v_B$ and $v_{A\cup B}$. Thus
\[
\frac{x_{A} - x_{A\cap B}}{z_A - z_{A\cap B}} < \frac{x_{A\cup B}-x_B}{z_{A\cup B}-z_B}.
\]
Notice that the numerators on both sides of this inequality are equal and positive, so we may rearrange terms to get
\[
z_A + z_B > z_{A\cup B} + z_{A\cap B},
\]
which violates the submodularity condition on the $z_I$. This is a contradiction. 

Now we may let $\pi = \pi_1|\cdots|\pi_k$ where $\pi_i = A_i \setminus A_{i-1}$. By construction $x$ satisfies the simplicial inequalities of $P^\pi(q)$, and by the increasing property of the $z_I$, $x$ satisfies the facial inequalities as well. Therefore $x \in P^\pi(q)$. 

Finally, note that if $x$ is generic then the partition $\pi$ is uniquely determined by the construction above. Therefore $P^{\pi}(q)$ and $P^{\pi'}(q)$ have disjoint interiors for $\pi \neq \pi'$.
\end{proof}

%
\begin{cor} The volume of the $q$-lifted polytope $P(q)$ is given by
\[
\vol_n(P(q)) = \sum_{\pi\in \mathcal{P}^n} 
\vol_n(P^\pi(q))
\]
\end{cor}

Motivated by this result, we now investigate the $\pi$-liftings $P^\pi(q)$ and their volumes in detail.

\begin{prop}
For $0 < q < 1$, the $\pi$-lifting $P^\pi(q)$ is combinatorially isomorphic to $\Delta_k \times P_\pi$. 
\end{prop}
\begin{proof} We prove the following stronger statement:

\begin{quote}
Suppose that, in the inequality description of $P^\pi(q)$ in Proposition \ref{prop: wedge ineqs}, we keep all the facial inequalities (\ref{facial}) and $t$ of the simplicial inequalities (\ref{simplicial}), and set the rest to be equalities. Then the resulting face $Q$ of $P^\pi(q)$ is combinatorially isomorphic to $\Delta_{t-1} \times P_\pi$.
\end{quote}

Notice that $t \geq 1$ since $q < 1$. First we prove the statement for $t=1$. Since $P$ is a generalized permutahedron, the $\pi$-maximal face $P_\pi = P_1 \times \cdots \times P_k$ for some polytopes $P_1 \subset \R^{\pi_1}, \ldots, P_k \subset \R^{\pi_k}$. If we set all but the $i$th facial inequality (\ref{facial}) to equalities, one easily checks that $Q = qP_1 \times \cdots \times  qP_{i-1} \times P_i \times \cdots \times P_k$. Since $q > 0$, $Q$ is combinatorially isomorphic to $P_\pi$.

Now we proceed by induction on $s:=\dim P_\pi+t$. The base case $s=1$ follows from the previous paragraph. Now consider a face $Q$ with $\dim P_\pi+t = s$. The facets of $Q$ are the following:

\noindent \textbf{Simplicial}: If $t=1$ then we already showed that $Q$ is isomorphic to $\Delta_0 \times P_\pi$. If $t \geq 1$ and we set any one of the remaining $t$ simplicial inequalities into an equality, the inductive hypothesis assures us that the result is isomorphic to $\Delta_{t-2} \times P_\pi$. 

\noindent \textbf{Facial}: Consider a facet of $Q$ given by an equation $x_{C_i}/z_{\pi'}^{C_i} = x_{D_i}/z_{\pi'}^{D_i}$. A vertex $v$ of $P \subset P^\pi(q)$ is on this facet if and only if $v \in P^{\pi'}(q)$. In turn, a ``$q$-lifting" of $v$ are on this facet if and only if $v$ is, since the lifting process applies a factor of $q$ to $v_{C_i}$  if and only if it applies it to $v_{D_i}$. Therefore this facet equals $P^{\pi'}(q)$, and is isomorphic to $\Delta_{t-1} \times P_{\pi'}$ by the inductive hypothesis.

From this it follows that $Q$ is combinatorially isomorphic to $\Delta_{t-1} \times P_{\pi}$, as we wished to show.
%
%
%
%
\end{proof}

\begin{theorem}\label{thm: integral}
Let $P$ be a generalized permutahedron in $\mathbb{R}^n$. Let $\pi=\pi_1|\cdots|\pi_k$ be an ordered partition of $[n]$. Then the volume of the $\pi$-lifting $P^\pi(q)$ is a polynomial in $q$ given by
\begin{align*}
\vol_n(P^\pi(q)) =  \frac{z_\pi}{\sqrt{|\pi_1|\cdots|\pi_k|}} \vol_{n-k}(P_\pi)\int_{q}^{1} \int_{q}^{t_k} \cdots \int_{q}^{t_2} t_1^{|\pi_1|-1}\cdots t_k^{|\pi_k|-1} dt_1\cdots dt_k,
\end{align*}
where $z_\pi = z_\pi^{\pi_1}\cdots z_\pi^{\pi_k}$. 

\end{theorem}
\begin{proof}
We use Federer's coarea formula \cite{Fe}. Consider the linear transformation
\begin{eqnarray*}
f: \R^n & \rightarrow & \R^k \\
x & \mapsto & \left(\frac{x_{\pi_1}}{z_\pi^{\pi_1}}, \ldots, \frac{x_{\pi_k}}{z_\pi^{\pi_k}}\right) 
\end{eqnarray*}
which maps $P^\pi(q)$ onto the $k$-simplex $\Delta:= \{y \in \R^k \, : \, q \leq y_1 \leq \cdots \leq y_k \leq 1\}$. One easily checks that the $k$-Jacobian of this map has norm ${\sqrt{|\pi_1|\cdots|\pi_k|}}/{z_\pi}$.

By Proposition \ref{prop:props}, the $\pi$-maximal face is of the form $P_\pi = P_1 \times \cdots \times P_k$ for some polytopes $P_1 \subset \R^{\pi_1}, \ldots, P_k \subset \R^{\pi_k}$. It is easy to see that
\[
f^{-1}(p)=(p_{\pi_1} \cdot P_1) \times \cdots \times  (p_{\pi_k} \cdot P_k)
\]
for any $p \in \Delta$. Therefore this fiber is combinatorially isomorphic to $P_{\pi}$ and
\[
\vol_{n-k}(f^{-1}(p)) =  p_{\pi_1}^{|\pi_1|-1}\cdots p_{\pi_k}^{|\pi_k|-1}\vol_{n-k}(P_\pi).
\]
The result follows by integrating this over $p \in \Delta$ and using the coarea formula.
\end{proof}

\begin{figure}[h]
\centering
\includegraphics[scale=.0815]{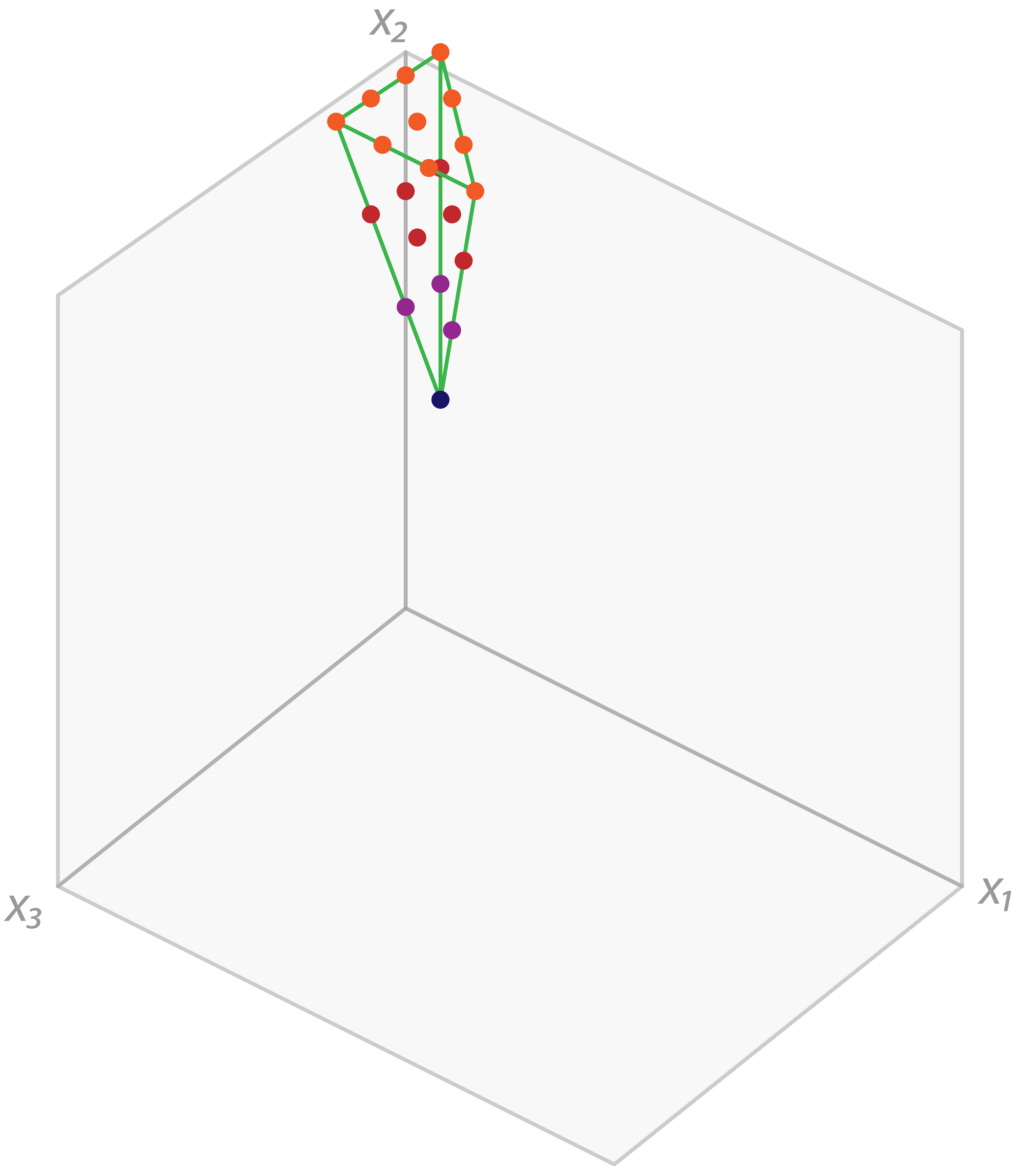} \includegraphics[scale=.0815]{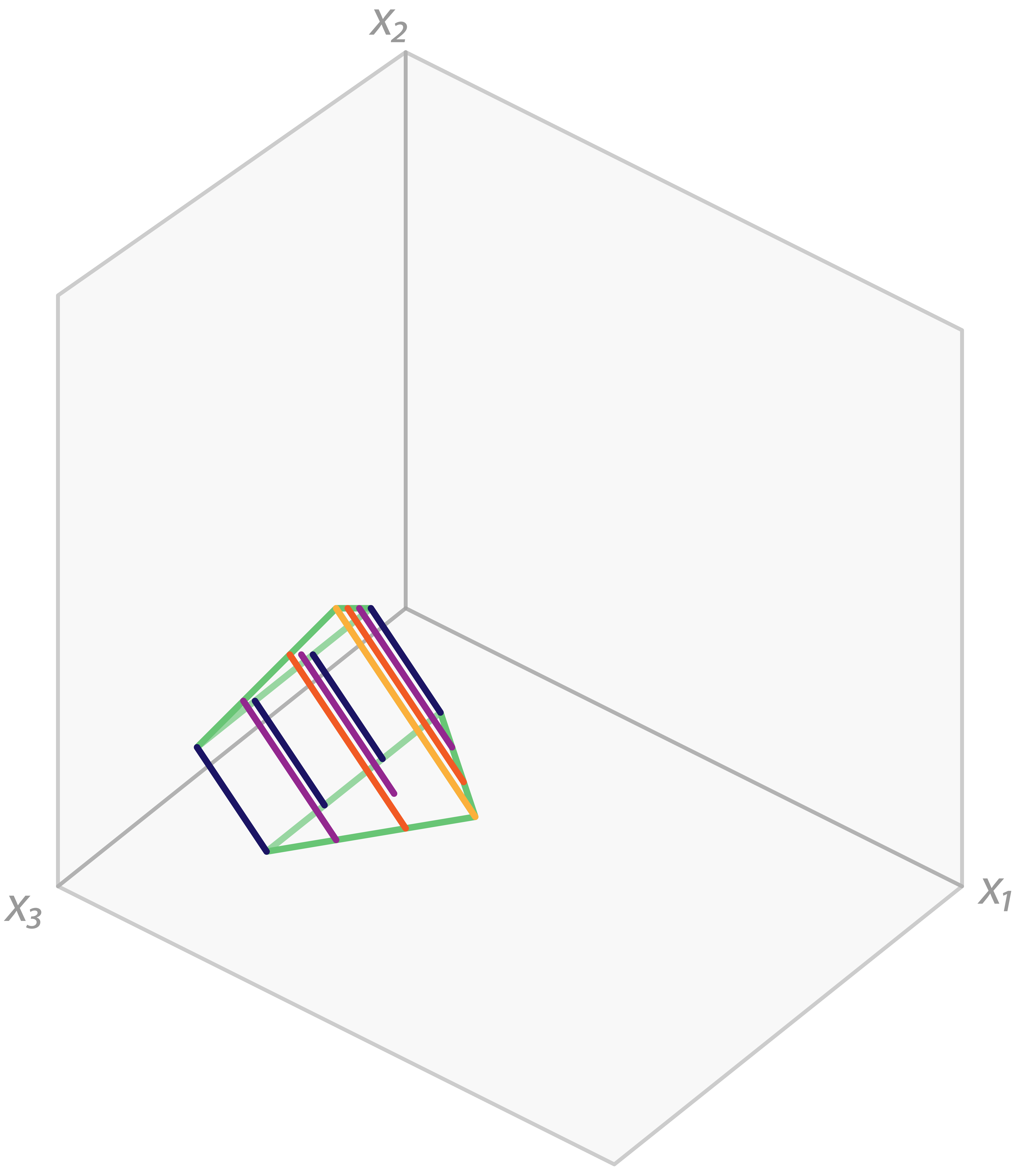} \includegraphics[scale=.0815]{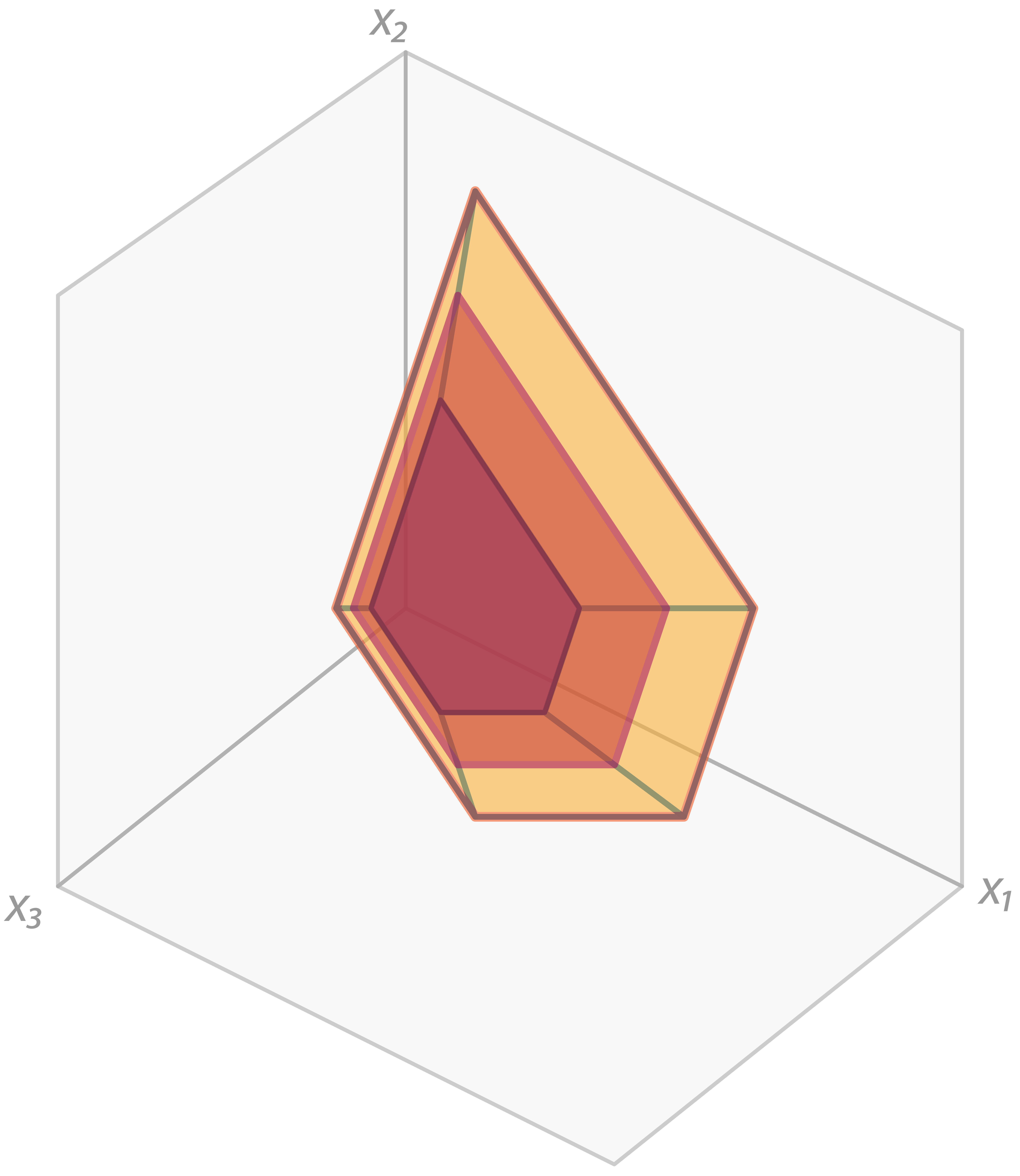} 
\caption{The $\pi$-liftings of the associahedron $\mathcal{K}(4)$, $\mathcal{K}(4)^{1|3|2}(q)$, $\mathcal{K}(4)^{12|3}(q)$, and $\mathcal{K}(4)^{123}(q)$ of Figure \ref{fig: face liftings}, together with some of the fibers that we are integrating to obtain their volume. The fibers are points, segments, and pentagons, respectively.}\label{fig: fibers}
\end{figure}


Observe that the above integral evaluates to a polynomial in $q$ and depends only on the sizes of the blocks of $\pi$. The sequence of these block sizes can be thought of as a composition $c(\pi)$ of the integer $n$. Let us call this polynomial $g_{c(\pi)}(q)$. This polynomial will be the subject of study of Part 2.

\bigskip
\bigskip
\bigskip

\newpage

\noindent
\begin{Large}
\textsf{PART 2. COMPOSITION POLYNOMIALS.}
\end{Large}
\bigskip


In Section \ref{section: composition polynomials}, motivated by the geometric considerations of Part 1, we introduce the \emph{composition polynomial} $g_c(q)$ of an ordered composition  $c = (c_1,\dots,c_k)$ of $n$ and the  \emph{reduced composition polynomial} $f_c(q) = (1-q)^{-k}g_c(q)$. We present our main results, Theorems \ref{theorem: g closed form} -- \ref{theorem: order polytope}

In Section \ref{section:formulas} we derive an explicit formula  (Theorem \ref{theorem: g closed form}) and various properties (Theorem \ref{theorem: composition polynomials}) of composition polynomials, and we prove the positivity of $f_c(q)$. (Theorem \ref{theorem: positivity})
%
In Section \ref{section:interpolation} we show that composition polynomials arise very naturally in the polynomial interpolation of the exponential function $h(x) = q^x$. (Theorem \ref{theorem: interpolation})
In Section \ref{section:orderpolytopes} we establish a connection between composition polynomials and Stanley's order polytopes. (Theorem \ref{theorem: order polytope}) We use this to interpret $g_c(q)$ as a generating function for counting linear extensions of a poset $P_c$.
We conclude by suggesting some questions in Section \ref{section: open}.


\section{\textsf{Composition polynomials.}}\label{section: composition polynomials}

\begin{deff}
A \emph{composition} $c=(c_1,\dots,c_k)$ is a finite ordered tuple of positive integers. We call the $c_i$ the \emph{parts} of $c$, and the sum $c_1+\cdots+c_k$ the \emph{size} of $c$. If $c = (c_1,\dots,c_k)$ has size $n$, we say that $c$ is a composition of $n$ into $k$ parts. The \emph{reverse} of the composition $c$ is defined as $\bar{c} = (c_k,\dots,c_1)$. 
\end{deff}

\begin{deff}
For a composition $c=(c_1,\dots,c_k)$ we write $\mathbf{t^{c-1}}: = t_1^{c_1-1}\cdots t_k^{c_k-1}$, where $t = (t_1,\dots,t_k)$. The \emph{composition polynomial} $g_c(q)$ is
\[
g_c(q) := \int_{q}^{1} \int_{q}^{t_k} \cdots \int_{q}^{t_2} \mathbf{t^{c-1}} dt_1\cdots dt_k. 
\]
The \emph{reduced composition polynomial} of $c$ is $f_c(q) = g_c(q)/(1-q)^k$. We will soon see in Theorem \ref{theorem: composition polynomials} that it is, indeed, a polynomial.
\end{deff}

It is clear that $g_c(q)$ is indeed a polynomial in $q$ of degree $n$. It is less clear that $f_c(q)$ is also a polynomial, but we will prove it in Theorem \ref{theorem: composition polynomials}. Below are some examples of composition polynomials which hint at some of their general properties.
 
\begin{itemize}
\item $g_{(1,1,1,1)}(q) = \frac{1}{24}(1-q)^4$.
\item $g_{(2,2,2,2)}(q) = \frac{1}{384}(1-q)^4(1+q)^4$.
\item $g_{(1,2,2)}(q) = \frac{1}{120}(1-q)^3(8 + 9q + 3q^2)$.
\item $g_{(2,2,1)}(q) = \frac{1}{120}(1-q)^3(3 + 9q + 8q^2)$.
\item $g_{(3,5)}(q) = \frac{1}{120}(1-q)^2(5+10q+15q^2+12q^3+9q^4+6q^5+3q^6)$.
\item $g_{(a,b)}(q) = \frac{1}{ab(a+b)}(1-q)^2(b+2bq+\cdots+(a-2)bq^{a-3}+(a-1)bq^{a-2}+ abq^{a-1}+ \\ 
+ a(b-1)q^a+ a(b-2)q^{a+1} + \cdots+2aq^{a+b-3} + aq^{a+b-2})$
\end{itemize}

For instance, the reader can check that $g_{(a,b)}(q) = \frac1{a(a+b)}(1-q^{a+b}) - \frac{q^a}{ab}(1-q^b)$, from which the last formula follows.



Our main results in Part 2 are the following:

\begin{theorem}\label{theorem: g closed form}
If $\beta_i = c_1 + \cdots + c_i$ for $0 \leq i \leq k$, we have
\[
g_c(q)=\sum_{i=0}^k \frac{q^{\beta_i}}{\prod_{j \neq i}(\beta_j - \beta_i)}.
\]
\end{theorem}

\begin{theorem}\label{theorem: composition polynomials}
Let $c = (c_1,\dots,c_k)$ be a composition of $n$. Then: 
\begin{enumerate}
\item $g_{\bar{c}}(q) = q^n g_c(1/q)$.
\item $g_{m c}(q) = \frac{1}{m^k}g_c(q^{m})$ 
for any positive integer $m$.
\item $g_c(q) = (1-q)^kf_c(q)$ for a polynomial $f_c(q)$ of degree $n-k$ with 
 $f_c(1)\neq 0$.
\item $f_c(1)=1/{k!}$.
\end{enumerate}
\end{theorem}

\begin{theorem}\label{theorem: positivity}
The coefficients of the reduced composition polynomial $f_c(q)$ are positive.
\end{theorem}

\begin{theorem}\label{theorem: interpolation}
Let $c=(c_1, \ldots, c_k)$ be a composition and let $\beta_i=c_1+\cdots+c_i$ for $i=0,\dots,k$. Let $h(x) = a_0 + a_1x + \cdots + a_kx^k$ be the polynomial of smallest degree that passes through the $k+1$ points $(\beta_i,q^{\beta_i})$. Here the coefficients $a_i$ are functions of $q$. Then $a_k = (-1)^k g_c(q)$.
\end{theorem}

\begin{theorem}\label{theorem: order polytope} 
There is a poset $P_c$ and an element $p \in P_c$ such that the volume of a slice of the order polytope $\mathcal{O}(P_c)$ in the $x_p$ direction is
\[
\vol({\mathcal{O}}(P_c) \cap (x_p=q)) = \frac{g_c(q)}{(c_1-1)! \cdots (c_k-1)!}.
\]
\end{theorem}
%
%

%
%
%
\section{\textsf{Recursive and explicit formulas} \label{section:formulas}}

\begin{deff}
Define the \emph{truncated compositions} $c^L := (c_2,\dots,c_k)$ and $c^R:=(c_1,\dots,c_{k-1})$. For $m \in \{1,\dots, k-1\}$ we define the \emph{merged composition} $c^m$ as the composition formed by combining the parts $c_m$ and $c_{m+1}$ into a single part:
\[
c^m := (c_1,\dots, c_{m-1},c_m+c_{m+1},c_{m+2},\dots,c_k).
\]
\end{deff}

\begin{lemma}\label{lemma: baby recursion}
For a composition $c=(c_1,\dots,c_k)$ of $n$, the composition polynomial $g_c(q)$ satisfies the recursion:
\[
g_c(q) = \frac{1}{c_1}g_{c^1}(q) - \frac{q^{c_1}}{c_1}g_{c^L}(q).
\]
\end{lemma}

\begin{proof}
We have:
\begin{align*}
g_c(q) &= \int_{q}^{1} \int_{q}^{t_k} \cdots \int_{q}^{t_2} t_1^{c_1-1}\cdots t_k^{c_k-1} dt_1\cdots dt_k \\
&=  \frac{1}{c_1}\int_{q}^{1} \int_{q}^{t_k} \cdots \int_{q}^{t_3} t_2^{c_2-1}\cdots t_k^{c_k-1} (t_2^{c_1} - q^{c_1})dt_2\cdots dt_k \\
&= \frac{1}{c_1}g_{(c_1+c_2,c_3,\dots,c_k)}(q) - \frac{q^{c_1}}{c_1}g_{(c_2,c_3,\dots,c_k)}(q) 
\end{align*}
as we wished to show. \end{proof}


Consider the sequence of partial sums $0=\beta_0<\cdots<\beta_k=n$ by $\beta_i=c_1+\cdots+c_i$ for $i=1,\dots,k$.  Let $(\beta)$ denote the Vandermonde matrix 
 \[
 (\beta) = \begin{pmatrix}
 1 & \beta_0 & \cdots & \beta_0^k \\
 \vdots & \vdots&  & \vdots \\
  1 & \beta_k & \cdots & \beta_k^k \\
 \end{pmatrix}.
\] 
We will index the rows and columns of this matrix from $0$ to $k$. Recall that 
\[
\text{det}(\beta) = \prod _{0\leq i < j \leq k} (\beta_j-\beta_i).
\]
For $0 \leq i \leq k$ let 
\[
[\beta_i] := (-1)^i \prod_{j\neq i} (\beta_j - \beta_i), \qquad [\hat{\beta}_i] := \text{det}(\beta)/[\beta_i]. 
\]
Notice that $[\hat{\beta}_i]$ is the unsigned minor of $(\beta)$ obtained by removing row $i$ and column $k$. Moreover, $[\hat{\beta}_i]$ is itself a Vandermonde determinant. We are ready to prove our explicit formula for composition polynomials, which we rewrite as:
%
%
%
\[
g_c(q)=\sum_{i=0}^k (-1)^i \frac{q^{\beta_i}}{[\beta_i]}.
\]
\begin{proof}[Proof of Theorem \ref{theorem: g closed form}]
Define $[\beta_i^R]$ analogously to $[\beta_i]$ for the truncated composition $c^R=(c_1,\dots,c_{k-1})$.
Proceed by induction on $k$. If $k=1$ then
\[
\int_{q}^{1}t_1^{c_1-1}dt_1 = \frac{1}{c_1} - \frac{q^{c_1}}{c_1} = \frac{q^{\beta_0}}{[\beta_0]} - \frac{q^{\beta_1}}{[\beta_1]}.
\]
Now assume that the formula holds up to $k-1$. Then 
\[
g_{c^R}(q)=\int_{q}^{1} \cdots \int_{q}^{t_2} \mathbf{t^{c^R-1}}dt_1\cdots dt_{k-1} =   \sum_{i=0}^{k-1} (-1)^i \frac{q^{\beta_i}}{[\beta^R_i]}.
\]
Changing the upper bound of the outer integral produces
\[
\int_{q}^{t_k} \cdots \int_{q}^{t_2} \mathbf{t^{c^R-1}}dt_1\cdots dt_{k-1} = \sum_{i=0}^{k-1} (-1)^i \frac{q^{\beta_i}t_k^{\beta_{k-1}-{\beta_i}}}{[\beta^R_i]}.
\]
This follows from the observation that this integral must evaluate to a homogeneous polynomial in $t_k$ and $q$ of total degree $c_1+\cdots+c_{k-1}= \beta_{k-1}$. The original integral we wish to compute becomes
\begin{align*}
g_c(q) &= \int_q^1 t_k^{c_k-1}\sum_{i=0}^{k-1} (-1)^i \frac{q^{\beta_i}t_k^{\beta_{k-1}-{\beta_i}}}{[\beta^R_i]} dt_k\\
&=  \sum_{i=0}^{k-1} (-1)^i q^{\beta_i} \int_q^1 \frac{t_k^{\beta_{k}-{\beta_i}-1}}{[\beta^R_i]}dt_k\\
&=\sum_{i=0}^{k-1} (-1)^i \frac{q^{\beta_i}}{[\beta_i]} - q^{\beta_k}\sum_{i=0}^{k-1} \frac{(-1)^i }{[\beta_i]}.
\end{align*}
Now observe that $(\beta)\sum_{i=0}^{k} {(-1)^i }/{[\beta_i]} = \sum_{i=0}^{k} {(-1)^i }{[\hat{\beta}_i]}$ computes, up to sign, the determinant of the matrix formed by replacing the last column in the Vandermonde matrix $(\beta)$ with a column of $1$s. This determinant is clearly zero, hence $-\sum_{i=0}^{k-1} {(-1)^i }/{[\beta_i]} = {(-1)^k}/{[\beta_k]}$. This gives us the desired result.
\end{proof}

\begin{cor}\label{cor: merged closed form}
Given a composition $c=(c_1,\dots,c_k)$, the composition polynomials of the associated merged and truncated compositions are given by
\begin{align*}
g_{c^m}(q)&=\sum_{i=0}^k (-1)^i \frac{q^{\beta_i}(\beta_m-\beta_i)}{[\beta_i]}, \\
g_{c^R}(q) &= \sum_{i=0}^k (-1)^i \frac{q^{\beta_i}(n-\beta_i)}{[\beta_i]}, \text{and}\\
q^{c_1}g_{c^L}(q) &= -\sum_{i=0}^k (-1)^i \frac{q^{\beta_i}\beta_i}{[\beta_i]}.
\end{align*}
\end{cor}
\begin{proof}
For the merged composition $c^m$, the partial sums $\beta^m_i$ are given by $\beta_i^m = \beta_i$ for $i<m$, and $\beta_i^m = \beta_{i+1}$ for $i\geq m$. From this observe that $[\beta^m_i] = [\beta_i]/(\beta_m-\beta_i)$ for $i<m$ and $[\beta^m_i] = [\beta_{i+1}]/(\beta_{i+1}-\beta_m)$ for $i\geq m$. Notice that the coefficient of $q^{\beta_m}$ is zero, as it should be. 

For the truncated composition $c^R$ the partial sums $\beta^R_i$ follow this same pattern. Finally, for the truncation $c^L$ we have $\beta^L_i = \beta_{i+1}-\beta_1$ for $i\geq 1$, and $\beta^L_0 = 0$. From this we observe that $[\beta^L_i] = [\beta_{i+1}]/\beta_{i+1}$ for all $i$. Substituting  into Theorem \ref{theorem: g closed form} yields the desired formulas.
\end{proof}

Now we can write down a stronger recursive formula for $g_{c}(q)$ that will be the key to our proof of Theorem \ref{theorem: positivity}.

\begin{cor}\label{cor: merged recursion}
Let $c=(c_1,\dots,c_k)$ be a composition of $n$ into $k$ parts. Let $c^m$ be the merged composition $(c_1,\dots,c_m+c_{m+1},\dots,c_k)$, and let $c^L = (c_2,\dots,c_k)$ and $c^R = (c_1,\dots,c_{k-1})$ be the truncated compositions. Then 
\begin{equation}\label{g recursion}
g_{c^m}(q) = \left(\frac{c_1+\cdots+c_m}{c_1+\cdots+c_k}\right)g_{c^R}(q)  + \left(\frac{c_{m+1}+\cdots+c_k}{c_1+\cdots+c_k}\right)q^{c_1}g_{c^L}(q).
%
\end{equation}
\end{cor}
\begin{proof}
This follows immediately from Corollary \ref{cor: merged closed form}.
\end{proof}

It is possible to write down several recursive equations for $g_c$, but this particular one is significant for several reasons:

\begin{itemize}
\item
Every non-trivial composition $c$ can be thought of as a merged composition for some $m$, and the sizes of $c^L$ and $c^R$ are each strictly less than the size of $c^m$. This means we have actually produced a recursive expression for an arbitrary nontrivial composition polynomial in terms of ``smaller" composition polynomials. This will allow us to prove Theorem \ref{theorem: composition polynomials} inductively.
\item
The compositions $c^m, c^L,$ and $c^R$ have the same length, so the polynomials $f_c$ turn out to satisfy exactly the same recursion as $g_c$ by Theorem \ref{theorem: composition polynomials}.4.
\item
Since this recursion only has positive terms, we will then obtain a proof of Theorem \ref{theorem: positivity}, the positivity of $f_c$.
\end{itemize}

\begin{proof}[{Proof of Theorems \ref{theorem: composition polynomials} and \ref{theorem: positivity}}]

%
Parts 1. and 2. of Theorem \ref{theorem: composition polynomials}  follow readily from our explicit formula for $g_c(q)$.
The partial sums of the reversal $\bar{c}$ are $\bar{\beta}_i = n - \beta_{k-i}$, and $[\bar{\beta}_i] = [\beta_{k-i}]$. The partial sums of $mc$ are $m\beta_i$, and $[m\beta_i] = m^k[\beta_i]$. Substituting these into Theorem \ref{theorem: g closed form} gives the results. 

We prove Theorems \ref{theorem: composition polynomials}.3, \ref{theorem: composition polynomials}.4, and \ref{theorem: positivity} by induction on the size of $c$ for a fixed $k$. The base case is  $c=(1,\dots,1)$, the composition of $k$ into $k$ parts. Theorem \ref{theorem: g closed form} gives
\[
g_{(1,\dots,1)}(q) = \sum_{i=0}^k (-1)^i \frac{q^{i}}{i!(k-i)!} = \frac{1}{k!}(1-q)^k.
\]
from which the claims follow readily.

Now suppose $c$ has size $n>k$. Then some part of $c$ is greater than $1$, and we can write $c$ as some merged composition $c'^m$ for some composition $c'$.
By Corollary \ref{cor: merged recursion}, 
\[
g_c(q) = g_{c'^m}(q) = \frac{\beta'_m}{n}g_{c'^R}(q)  + \left(1-\frac{\beta'_m}{n}\right)q^{c'_1}g_{c'^L}(q).
\]
Notice that $c'^R$ and $c'^L$ are compositions of length $k$ and size strictly smaller than $c$. Therefore by induction we may write
\begin{align*}
g_c(q) &= \frac{\beta'_m}{n}(1-q)^k f_{c'^R}(q)  + \left(1-\frac{\beta'_m}{n}\right)q^{c'_1}(1-q)^k f_{c'^L}(q) \\
&=(1-q)^k\left(\frac{\beta'_m}{n} f_{c'^R}(q)  + \left(1-\frac{\beta'_m}{n}\right)q^{c'_1} f_{c'^L}(q)\right) \\
&=: (1-q)^kf_c(q).
\end{align*}
where $f_c(q)$ is a polynomial of degree $n-k$. Since $\frac{\beta'_m}{n}$ and $\left(1-\frac{\beta'_m}{n}\right)$ are positive and they sum to 1, $f_c$ inherits the desired properties from $f_{c'^R}$ and $f_{c'^L}$.
\end{proof}

Further examples seem to suggest that the sequence of coefficients of $f_c(q)$ is unimodal, meaning that the coefficients $f_c(q) = \sum_{i = 0}^{n-k}f_iq^i$ satisfy the inequalities $f_1 \leq f_2 \leq \cdots \leq f_{i-1} \leq f_i \geq f_{i+1} \geq \cdots \geq f_{n-k}$ for some $i$. More strongly, the sequence may even be log-concave, meaning that $f_j^2 \geq f_{j-1}f_{j+1}$ for all $j$. We have verified both statements for all 335,922 compositions of at most 7 parts and sizes of parts at most 6.
\begin{quest}\label{question}
Is the sequence of coefficients of $f_c(q)$ always unimodal?
\end{quest}

Since $g_c(q)$ is essentially the volume of a Minkowski sum of two polytopes (Proposition \ref{prop: face decomp}), one might hope to derive the log-concavity of the $f_i$ from the Aleksandrov-Fenchel inequalities \cite{StLog, StFenchel}. The ``obvious" application of these inequalities does not seem to give the desired result, and the question remains open.

We conclude this section with an explicit formula for the coefficients of $f_c(q)$. Unfortunately, this formula does not seem to explain their unimodality, or even their positivity (Theorem \ref{theorem: positivity}). 
Recall the notation
\[
\left(\!\!\binom{n}{k}\!\! \right) := {{n+k-1}\choose{k-1}}
\]
for the number of multisets of $[n]$ of size $k$. 

\begin{cor}
The reduced composition polynomial $f_c(q)=\sum_{i=0}^{n-k}f_iq^i$ has
\[
f_i = \sum_{j:\beta_j\leq i} \frac{(-1)^{j}}{[\beta_j]} 
\left(\!\!\binom{k}{i-\beta_j}\!\!\right).
\]
\end{cor}

\begin{proof}
We compute
\begin{align*}
f_c(q) &= g_c(q)(1+q+q^2+\cdots)^k \\
&= \left(\sum_{j=0}^k(-1)^j\frac{q^{\beta_j}}{[\beta_j]}\right)\left(\sum_{i=0}^\infty  \left(\!\!\binom{k}{i}\!\!\right)q^i \right)
\end{align*}
as desired. \end{proof}


\section{\textsf{Composition polynomials in polynomial interpolation.}}\label{section:interpolation}

Now we prove Theorem \ref{theorem: interpolation}, which shows that composition polynomials have a very natural interpretation in terms of the polynomial interpolation of an exponential function $e(x)=q^x$. 

Recall that $h(x)$ is the polynomial of smallest degree which agrees with $e(x)=q^x$ at the points $\beta_i = c_1 + \cdots +c_i$. We wish to show that the leading coefficient of $h(x)$, which is a function of $q$, in fact equals $(-1)^k g_c(q)$.

\begin{proof}[Proof of Theorem \ref{theorem: interpolation}]
Theorem \ref{theorem: g closed form} implies that $\text{det}(\beta)g_c(q) = \sum_{i=0}^k (-1)^i q^{\beta_i} [\hat{\beta}_i]$, which we rewrite as

\begin{equation}\label{vandermonde}
 \det(\beta) \cdot  g_c(q) = (-1)^k\text{det}\begin{pmatrix}
 1 & \beta_0 & \cdots & \beta_0^{k-1} & q^{\beta_0} \\
 \vdots & \vdots& \cdots & \vdots & \vdots \\
  1 & \beta_k & \cdots & \beta_k^{k-1} & q^{\beta_k} \\
 \end{pmatrix}.
\end{equation}

Now notice that this is, up to sign, precisely what we obtain when we use Cramer's rule to solve for $a_k$ in the system of linear equations
\[
\begin{pmatrix}
 1 & \beta_0 & \cdots & \beta_0^{k}  \\
 \vdots & \vdots& \cdots & \vdots \\
  1 & \beta_k & \cdots & \beta_k^{k}  \\
 \end{pmatrix}
 \begin{pmatrix}
 a_0  \\
 \vdots  \\
 a_k \\
 \end{pmatrix}
 =
 \begin{pmatrix}
 q^{\beta_0}  \\
 \vdots  \\
 q^{\beta_k} \\
 \end{pmatrix}.
 \]
But this system is equivalent to the polynomial interpolation problem under consideration. The desired result follows. 
\end{proof}

We can also interpret the individual coefficients of $f_c(q)$ in terms of the polynomial interpolation of a polynomial function which has been ``shut off" after $q=i$.
Consider the function
\[
d(x) = \begin{cases}  \big(\!\binom{k}{i-x}\!  \big), & \mbox{ if } x\leq i \\ 0, & \mbox{ if } x > i \end{cases}
\]

\begin{prop}
Let $f_c(q)= \sum_{i=0}^{n-k}f_iq^i$. Then $(-1)^k f_i$ is the lead coefficient of the polynomial $p_i(x)$ of smallest degree that passes through the points $(\beta_j, h(\beta_j))$ for $j=0, 1, \ldots, k$.
\end{prop}
\begin{proof}
This follows from a similar argument. 
\end{proof}

We can use these results to give non-recursive explanations of parts of Theorem \ref{theorem: composition polynomials}. We need a simple lemma.


\begin{lemma}\label{lem: vandermonde}
Let $(\beta)^{p}$ be  the matrix formed from the Vandermonde matrix $(\beta)$ by replacing the entries $\beta_i^k$ of the last column of $(\beta)$ with a polynomial $p(\beta_i)$ of degree $d\leq k$ and lead coefficient $c$. Then 
\[
\det ((\beta)^p) = \begin{cases} 0&  \textrm{ if } d<k, \\ c\cdot \text{det}(\beta) & \textrm{ if } d=k.\end{cases}
\] 
\end{lemma}
\begin{proof}
For $d=k$ we simply observe that $(\beta)^p$ can be obtained from $(\beta)$ via elementary column operations. The only such operation that affects the determinant is multiplying the last column of $(\beta)$ by $c$. If $d<k$ then the last column of $(\beta)^p$ is a linear combination of the previous columns, and thus the matrix is singular.
\end{proof}

\begin{proof}
[Alternate proof of Theorem \ref{theorem: composition polynomials}.3 and \ref{theorem: composition polynomials}.4.]\label{prop: alt1}
Taking the $i^{th}$ derivative of \eqref{vandermonde} gives
%
\[
\det(\beta) g_c^{(i)}(1) = (-1)^k\text{det}\begin{pmatrix}
 1 & \beta_0 & \cdots & \beta_0^{k-1} & \beta_0(\beta_0-1)\cdots (\beta_0-i+1) \\
 \vdots & \vdots& \cdots & \vdots & \vdots \\
  1 & \beta_k & \cdots & \beta_k^{k-1} & \beta_k(\beta_k-1)\cdots (\beta_k-i+1) \\
 \end{pmatrix}.
\] 
Lemma \ref{lem: vandermonde} tells us that this equals $0$ for $0 \leq i \leq k-1$ and $(-1)^k\det(\beta)$ for $i=k$. Therefore $1$ is a root of order $k$ in $g_c(q)$, and taking the $k^{th}$ derivative of  $g_c(q)=(1-q)^kf_c(q)$ we obtain $f_c(1) = \frac1{k!}$.
%
\end{proof}

\section{\textsf{Composition polynomials and order polytopes}}\label{section:orderpolytopes}

Consider the poset $P_c$ consisting of a chain $p_0 < p_1 < \cdots < p_k$ together with a chain of size $c_i-1$ below $p_i$ for $1 \leq i \leq k$. The \emph{order polytope} ${\mathcal O}(P_c)$, introduced by Stanley in \cite{Stanley}, is the polytope of points $x \in \mathbb{R}^{P_c}$ such that $0 \leq x_i \leq x_j \leq 1$ whenever $i \leq j \in P$.

\begin{figure}[h]
\centering
\includegraphics[scale=.2]{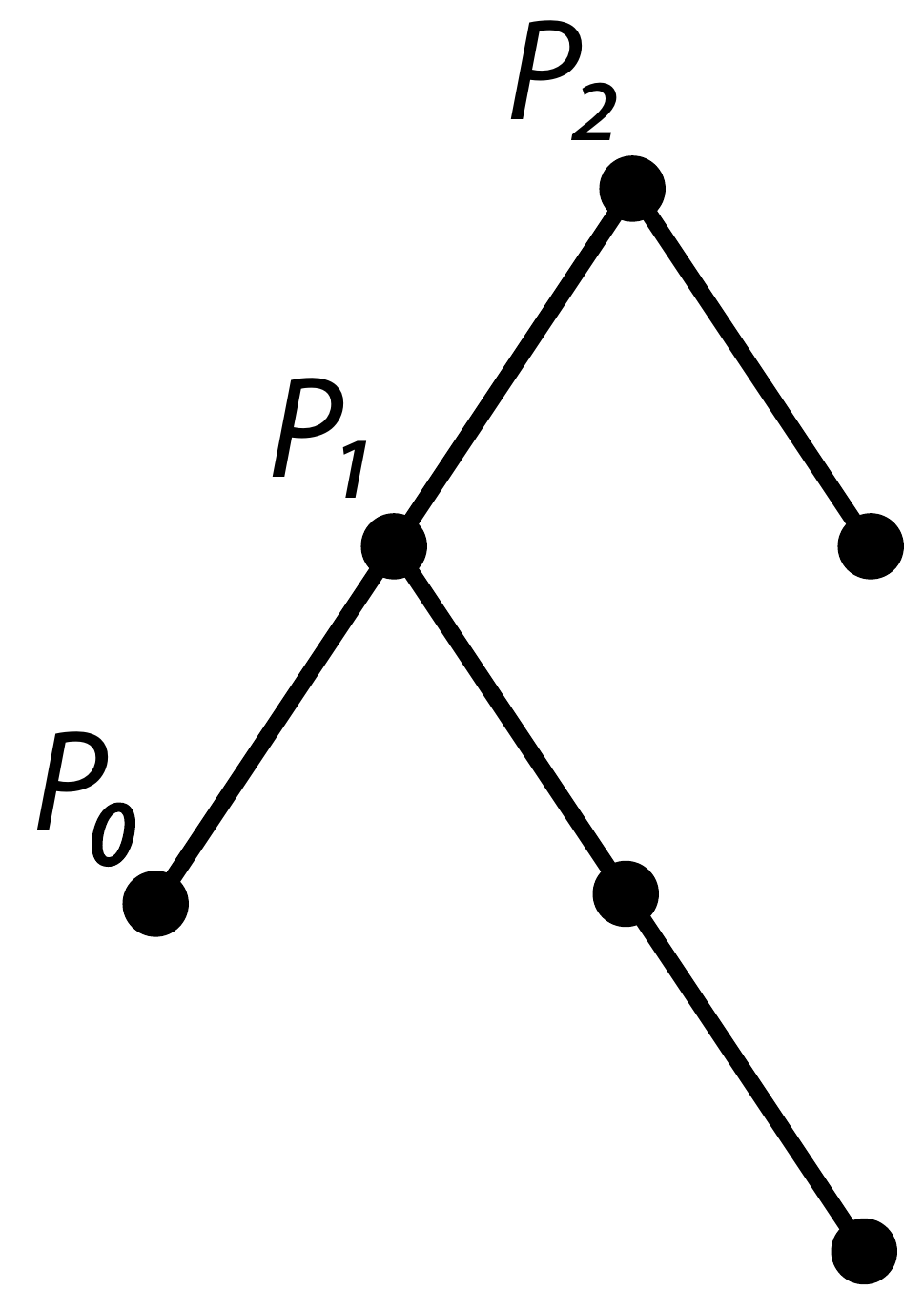} 
\caption{The poset $P_{32}$. \label{fig:P_32}}
\end{figure}

\begin{prop} \label{prop:order}
Let $H \in \mathbb{R}^{P_c}$ be the hyperplane $x_{p_0} = q$. Then
\[
\vol({\mathcal O}(P_c) \cap H) = \frac{g_c(q)}{(c_1-1)! \cdots (c_k-1)!}.
\]
\end{prop}

\begin{proof}
For any $0 \leq q \leq t_1 \leq \cdots \leq \cdots t_k \leq 1$, the intersection of ${\mathcal O}(P_c)$ with $x_{p_0}=q$ and $x_{p_i} = t_i$ for $1 \leq i \leq k$ is a product of $k$ simplices having volume $\prod_{i=1}^k \frac{t_i^{c_i-1}}{(c_i-1)!}$. Now integrate over all such values.
\end{proof}

\begin{cor} The composition polynomial is given by
\[
 g_c(q) =\frac{(c_1-1)! \cdots (c_k-1)!}{n!}  \sum_{i=0}^n N_{i+1} {n \choose i} q^i (1-q)^{n-i}
\]
where $N_j$ is the number of linear extensions of $P_c$ such that $x_0$ has height $j$. We have $N_j^2 \geq N_{j-1}N_{j+1}$ for  $2 \leq j \leq n$.
\end{cor}

\begin{proof}
This follows from Stanley's work on order polytopes, namely Proposition \ref{prop:order} and (15) of \cite{Stanley}.
\end{proof}

\section{\textsf{Questions and further directions}}\label{section: open}

Our work raises the following questions.

\begin{itemize}
\item
Find a simple combinatorial interpretation of 
the coefficients of $f_c(q)$.
\item
Our proof of Theorem \ref{theorem: interpolation} does not really explain the connection between the polytopes we study and the fundamental problem of interpolating an exponential function by polynomials. Find a more conceptual proof.
\item
Settle Question \ref{question}: Are the coefficients of $f_c(q)$ unimodal? Are they log-concave? 
\item
Describe the combinatorics of the liftings of other generalized permutahedra of interest, such as Hohlweg and Lange's realizations of the associahedron \cite{HL} or Pilaud and Santos's brick polytopes. \cite{PSa}
\end{itemize}

\section{\textsf{Acknowledgments}}

We thank Stefan Forcey and Pablo Schmerkin for valuable discussions, and an anonymous referee for a very thorough and helpful report. 

\footnotesize{

}


\begin{thebibliography}{99}


\bibitem{AA}
M. Aguiar, F. Ardila. \textit{The Hopf monoid of generalized permutahedra}. Preprint, 2011.

\bibitem{ABD}
F. Ardila, C. Benedetti, J. Doker. \textit{Matroid Polytopes and their Volumes.} Discrete \& Computational Geometry \textbf{43} (2010)

\bibitem{ARW}
F. Ardila, V. Reiner, and L. Williams. \emph{Bergman complexes, Coxeter arrangements, and graph associahedra. } Seminaire Lotharingien de Combinatoire, 54A (2006), Article B54Aj.

\bibitem{BGS}
A. Borovik, I. Gelfand, and N. White. \emph{Coxeter matroids}. Birkh\"auser, Boston, 2003.

\bibitem{CD}
M. Carr, S. Devadoss, \emph{Coxeter complexes and graph associahedra.}
Topology and its Applications 153 (2006) 2155-2168 

\bibitem{DF}
Devadoss and Forcey. \textit{Marked tubes and the graph multiplihedron.} 
Algebraic and Geometric Topology, 8(4) 2081-2108, 2008.

\bibitem{Fe}
H. Federer, \textit{Curvature measures}, Transactions of the American Mathematical Society {\bf 93} (1959) 418Ð491.

\bibitem{FS}
Feichtner and Sturmfels, \textit{Matroid polytopes, nested sets and Bergman fans.} Port. Math. (N.S.) 62 (2005), 437-468.

\bibitem{F}
S. Forcey. \textit{Convex Hull Realizations of the Multiplihedra.} Topology and Its Applications \textbf{156}, no. 2 (2008), 326--347.

\bibitem{FOOO}
K. Fukaya, Y. Oh, H. Ohta, and K. Ono. Lagrangian Intersection Floer Theory: Anomaly and Obstruction. AMS/IP Studies in Advanced Mathematics, Vol. 46, 2009

\bibitem{Fujishige}
S. Fujishige, Submodular functions and optimization, second ed., Annals of Discrete Mathematics, vol. 58, Elsevier B. V., Amsterdam, 2005

\bibitem{HL}
C. Hohlweg, C. Lange, \emph{Realizations of the associahedron and cyclohedron} Discrete and Computational Geometry \textbf{37} (2007) 517-543.

\bibitem{MW}
S. Mau and C. Woodward, Geometric realizations of the multiplihedron and its complexification. Preprint, 2008. \texttt{arXiv:0802.2120}.

\bibitem{Morton}
J. Morton, Lior Pachter, Anne Shiu, Bernd Sturmfels, Oliver Wienand. \textit{Convex Rank Tests and Semigraphoids.} Siam Journal on Discrete Mathematics \textbf{23} (2009), 1117--1134.

\bibitem{PSa}
V. Pilaud and F. Santos. \emph{The brick polytope of a sorting network}
European Journal of Combinatorics {\bf 33} (2012) 632-662.

\bibitem{PS}
J. Pitman and R. Stanley. \emph{A polytope related to empirical dis- tribution, plane trees, parking functions, and the associahedron.} Discrete Comput. Geom. {\bf 27}  {2002} 603-632.

\bibitem{Po}
A. Postnikov. \textit{Permutohedra, associahedra and beyond.}
Int. Math. Res. Notices \textbf{2009} (6) (2009), 1026--1106.

\bibitem{Po2}
A. Postnikov. Personal communication. 2007.

\bibitem{PRW}
A. Postnikov, V. Reiner, L. Williams. \textit{Faces of generalized permutahedra.}
Documenta Mathematica, \textbf{13} (2008), 207--273.

\bibitem{Sc}
A. Schrijver. Combinatorial optimization. Polyhedra and efficiency. Algorithms and Combinatorics {\bf 24}, Springer-Verlag, Berlin, 2003.

\bibitem{StLog}
R. P. Stanley. \textit{Log-Concave and Unimodal Sequences in Algebra, Combinatorics, and Geometry.} Annals of the New York Academy of Sciences, \textbf{576} (1989), 500-Ð535. 

\bibitem{Stanley}
Richard P. Stanley. \textit{Two poset polytopes}. Discrete Comput. Geom. {\bf 1} (1986) 9Ð23.

\bibitem{StFenchel}
R. P. Stanley. \textit{Two Combinatorial Applications of the Aleksandrov-Fenchel Inequalities.} J. Comb. Theory, Ser. A (1981), 56--65.

\bibitem{S}
J. Stasheff. H-spaces from a homotopy point of view. Lecture Notes in Mathematics, Vol. 161. Springer- Verlag, Berlin, 1970.

\bibitem{Zi}
G. Ziegler, \emph{Lectures on polytopes}, Graduate Texts in Mathematics,
vol. 152, Springer-Verlag, New York, 1995.

\end{thebibliography}
\end{document}